\pdfoutput=1
\RequirePackage{ifpdf}
\ifpdf % We are running pdfTeX in pdf mode
\documentclass[pdftex]{sigma}
\else
\documentclass{sigma}
\fi

\usepackage[all,cmtip]{xy}
\usepackage{framed}
\usepackage{comment}
\usepackage{needspace}
\usepackage{tabu}

\newtheorem{Theorem}{Theorem}[section]
\newtheorem{Corollary}[Theorem]{Corollary}
\newtheorem{Lemma}[Theorem]{Lemma}
\newtheorem{Proposition}[Theorem]{Proposition}
\newtheorem{Claim}[Theorem]{Claim}
 { \theoremstyle{definition}
\newtheorem{Definition}[Theorem]{Definition}
\newtheorem{Example}[Theorem]{Example}
\newtheorem{Remark}[Theorem]{Remark}
\newtheorem{question}[Theorem]{Question} }

\numberwithin{equation}{section}

\newcommand{\calX}{\mathcal{X}}

\newcommand{\calQ}{\mathcal{Q}}
\newcommand{\calT}{\mathcal{T}}

\newcommand\arcstr{\ar@/^1pc/}
\newcommand{\arinj}{\ar@{_{(}->}}
\let\sumnonlimits\sum
\let\prodnonlimits\prod
\renewcommand{\sum}{\sumnonlimits\limits}
\renewcommand{\prod}{\prodnonlimits\limits}

\begin{document}
\allowdisplaybreaks

\newcommand{\arXivNumber}{1704.00839}

\renewcommand{\PaperNumber}{078}

\FirstPageHeading

\ShortArticleName{$t$-Unique Reductions for M\'{e}sz\'aros's Subdivision Algebra}

\ArticleName{$\boldsymbol{t}$-Unique Reductions\\ for M\'{e}sz\'aros's Subdivision Algebra}

\Author{Darij GRINBERG}

\AuthorNameForHeading{D.~Grinberg}

\Address{School of Mathematics, University of Minnesota,\\
206 Church St. SE, Minneapolis, MN 55455, USA}
\Email{\href{mailto:darijgrinberg@gmail.com}{darijgrinberg@gmail.com}}
\URLaddress{\url{http://www.cip.ifi.lmu.de/~grinberg/}}

\ArticleDates{Received November 22, 2017, in final form July 15, 2018; Published online July 26, 2018}

\Abstract{Fix a commutative ring $\mathbf{k}$, two elements $\beta \in\mathbf{k}$ and $\alpha\in\mathbf{k}$ and a positive integer $n$. Let $\mathcal{X}$ be the polynomial ring over $\mathbf{k}$ in the $n(n-1) /2$ indeterminates $x_{i,j}$ for all $1\leq i<j\leq n$. Consider the ideal $\mathcal{J}$ of $\mathcal{X}$ generated by all polynomials of the form $x_{i,j}x_{j,k}-x_{i,k} ( x_{i,j}+x_{j,k}+\beta ) -\alpha$ for $1\leq i<j<k\leq n$. The quotient algebra $\mathcal{X}/\mathcal{J}$ (at least for a certain choice of $\mathbf{k}$, $\beta$ and $\alpha$) has been introduced by Karola M\'{e}sz\'{a}ros in [\textit{Trans. Amer. Math. Soc.} \textbf{363} (2011), 4359--4382] as a~commutative analogue of Anatol Kirillov's quasi-classical Yang--Baxter algebra. A monomial in $\mathcal{X}$ is said to be \textit{pathless} if it has no divisors of the form~$x_{i,j}x_{j,k}$ with $1\leq i<j<k\leq n$. The residue classes of these pathless monomials span the $\mathbf{k}$-module~$\mathcal{X}/\mathcal{J}$, but (in general) are $\mathbf{k}$-linearly dependent. More combinatorially: reducing a~given $p\in\mathcal{X}$ modulo the ideal~$\mathcal{J}$ by applying replacements of the form $x_{i,j}x_{j,k}\mapsto x_{i,k} ( x_{i,j}+x_{j,k}+\beta ) +\alpha$ always eventually leads to a $\mathbf{k}$-linear combination of pathless monomials, but the result may depend on the choices made in the process. More recently, the study of Grothendieck polynomials has led Laura Escobar and Karola M\'{e}sz\'{a}ros [\textit{Algebraic Combin.} \textbf{1} (2018), 395--414] to defining a $\mathbf{k}$-algebra homomorphism $D$ from $\mathcal{X}$ into the polynomial ring $\mathbf{k} [ t_{1},t_{2},\ldots,t_{n-1} ] $ that sends each $x_{i,j}$ to $t_{i}$. We show the following fact (generalizing a conjecture of M\'{e}sz\'{a}ros): If $p\in\mathcal{X}$, and if $q\in\mathcal{X}$ is a~$\mathbf{k}$-linear combination of pathless monomials satisfying $p\equiv q\operatorname{mod}\mathcal{J}$, then $D(q) $ does not depend on~$q$ (as long as~$\beta$,~$\alpha$ and~$p$ are fixed). Thus, the above way of reducing a $p\in\mathcal{X}$ modulo $\mathcal{J}$ may lead to different results, but all of them become identical once~$D$ is applied. We also find an actual basis of the $\mathbf{k}$-module $\mathcal{X}/\mathcal{J}$, using what we call \textit{forkless monomials}.}

\Keywords{subdivision algebra; Yang--Baxter relations; Gr\"obner bases; Arnold relations; Orlik--Terao algebras; noncommutative algebra}

\Classification{05E15; 05E40}

\tableofcontents

\section{Introduction}

The main result of this paper is probably best illustrated by an example:

\begin{Example}
\label{exam.intro}Let us play a solitaire game. Fix a positive integer $n$ and two numbers $\beta\in\mathbb{Q}$ and $\alpha\in\mathbb{Q}$, and let $\mathcal{X}$ be the ring $\mathbb{Q} [ x_{i,j}\,|\, 1\leq i<j\leq n ] $ of polynomials with rational coefficients in the $n(n-1) /2$ indeterminates $x_{i,j}$ with $1\leq i<j\leq n$. (For example, if $n=4$, then $\mathcal{X}=\mathbb{Q} [ x_{1,2},x_{1,3},x_{1,4},x_{2,3},x_{2,4},x_{3,4}] $.)

Start with any polynomial $p\in\mathcal{X}$. The allowed move is the following: Pick a monomial $\mathfrak{m}$ that appears (with nonzero coefficient) in $p$ and that is divisible by $x_{i,j}x_{j,k}$ for some $1\leq i<j<k\leq n$. For example, $x_{1,2}x_{1,3}x_{2,4}$ is such a monomial (if it appears in $p$ and if $n\geq4$), because it is divisible by $x_{i,j}x_{j,k}$ for $( i,j,k) =( 1,2,4) $. Choose one triple $(i,j,k) $ with $1\leq i<j<k\leq n$ and $x_{i,j}x_{j,k} \,|\, \mathfrak{m}$ (sometimes, there are several choices). Now, replace this monomial~$\mathfrak{m}$ by $\frac{x_{i,k}( x_{i,j}+x_{j,k}+\beta) +\alpha}{x_{i,j}x_{j,k}}\mathfrak{m}$ in~$p$.

Thus, each move modifies the polynomial, replacing a monomial by a sum of four monomials (or fewer, if $\beta$ or $\alpha$ is $0$). The game ends when no more moves are possible (i.e., no mono\-mial~$\mathfrak{m}$ appearing in your polynomial is divisible by $x_{i,j}x_{j,k}$ for any $1\leq i<j<k\leq n$).

It is easy to see that this game (a thinly veiled reduction procedure modulo an ideal of $\mathcal{X}$) always ends after finitely many moves. Here is one instance of this game being played, for $n=4$ and $\beta=1$ and $\alpha=0$ and starting with the polynomial $p=x_{1,2}x_{2,3}x_{3,4}$:
\begin{gather}
 x_{1,2}x_{2,3}x_{3,4} \mapsto x_{1,3}\left( x_{1,2}+x_{2,3}+1\right) x_{3,4}\nonumber\\
 \qquad\qquad ( \text{here, we chose }\mathfrak{m} =x_{1,2}x_{2,3}x_{3,4}\text{ and }(i,j,k) =( 1,2,3)) \nonumber\\
 \qquad{} =x_{1,2}x_{1,3}x_{3,4}+x_{1,3}x_{2,3}x_{3,4}+x_{1,3}x_{3,4}\nonumber\\
 \quad{} \mapsto x_{1,2}x_{1,4}( x_{1,3}+x_{3,4}+1) +x_{1,3}x_{2,3}x_{3,4}+x_{1,3}x_{3,4}\nonumber\\
\qquad\qquad ( \text{here, we chose }\mathfrak{m}=x_{1,2}x_{1,3}x_{3,4}\text{ and }(i,j,k) =( 1,3,4)) \nonumber\\
 \qquad{} =x_{1,2}x_{1,3}x_{1,4}+x_{1,2}x_{1,4}x_{3,4}+x_{1,2}x_{1,4}+x_{1,3}x_{2,3}x_{3,4}+x_{1,3}x_{3,4}\nonumber\\
 \quad{} \mapsto x_{1,2}x_{1,3}x_{1,4}+x_{1,2}x_{1,4}x_{3,4}+x_{1,2}x_{1,4}+x_{1,3}x_{2,4}( x_{2,3}+x_{3,4}+1) +x_{1,3}x_{3,4}\nonumber\\
\qquad\qquad (\text{here, we chose }\mathfrak{m}=x_{1,3}x_{2,3}x_{3,4}\text{ and }(i,j,k) =( 2,3,4)) \nonumber\\
 \qquad{} =x_{1,2}x_{1,3}x_{1,4}+x_{1,2}x_{1,4}x_{3,4}+x_{1,2}x_{1,4}+x_{1,3} x_{2,3}x_{2,4}+x_{1,3}x_{2,4}x_{3,4}\nonumber\\
 \qquad\quad{} +x_{1,3}x_{2,4} +x_{1,3}x_{3,4}\nonumber\\
 \quad{} \mapsto x_{1,2}x_{1,3}x_{1,4}+x_{1,2}x_{1,4}x_{3,4}+x_{1,2}x_{1,4} +x_{1,3}x_{2,3}x_{2,4}+x_{1,3}x_{2,4}x_{3,4}+x_{1,3}x_{2,4}\nonumber\\
\qquad\quad{} +x_{1,4} ( x_{1,3}+x_{3,4}+1 ) \nonumber\\
\qquad\qquad ( \text{here, we chose }\mathfrak{m}=x_{1,3}x_{3,4}\text{ and }(i,j,k) =( 1,3,4)) \nonumber\\
 \qquad{} =x_{1,2}x_{1,3}x_{1,4}+x_{1,2}x_{1,4}x_{3,4}+x_{1,2}x_{1,4}+x_{1,3}x_{2,3}x_{2,4}+x_{1,3}x_{2,4}x_{3,4}\nonumber\\
\qquad\quad{} +x_{1,3}x_{2,4}+x_{1,3}x_{1,4}+x_{1,4}x_{3,4}+x_{1,4}\nonumber\\
 \quad{} \mapsto x_{1,2}x_{1,3}x_{1,4}+x_{1,2}x_{1,4}x_{3,4}+x_{1,2}x_{1,4}+x_{1,3}x_{2,3}x_{2,4}+x_{2,4}x_{1,4} ( x_{1,3}+x_{3,4}+1 )\nonumber\\
\qquad\quad{} +x_{1,3}x_{2,4}+x_{1,3}x_{1,4}+x_{1,4}x_{3,4}+x_{1,4}\nonumber\\
\qquad\qquad ( \text{here, we chose }\mathfrak{m}=x_{1,3}x_{2,4}x_{3,4}\text{ and }(i,j,k) =( 1,3,4)) \nonumber\\
 \qquad{} =x_{1,2}x_{1,3}x_{1,4}+x_{1,2}x_{1,4}x_{3,4}+x_{1,2}x_{1,4}+x_{1,3}x_{2,3}x_{2,4}+x_{1,3}x_{1,4}x_{2,4}\nonumber\\
\qquad\quad{}+x_{1,4}x_{2,4}x_{3,4}+x_{1,4}x_{2,4}+x_{1,3}x_{2,4}+x_{1,3}x_{1,4}+x_{1,4}x_{3,4}+x_{1,4}. \label{eq.exam.intro.res1}
\end{gather}
The game ends at this polynomial, since there are no more moves to be done.

A standard question about games like this is: Is the state obtained at the end of the game (i.e., in our case, the polynomial after the game has ended) independent of the choices made during the game? In our case, the answer is ``no'' (in general, for $n\geq4$). Indeed, the reader can easily verify that the above game could have led to a different result if we had made different choices.

However, something else turns out to be independent of the choices. Namely, let us transform the polynomial at the end of the game further by applying the substitution $x_{i,j}\mapsto t_{i}$ (where $t_{1},t_{2},\ldots,t_{n-1}$ are new indeterminates). For example, doing this to the polynomial~(\ref{eq.exam.intro.res1}) results in%
\begin{gather*}
 t_{1}t_{1}t_{1}+t_{1}t_{1}t_{3}+t_{1}t_{1}+t_{1}t_{2}t_{2}+t_{1}t_{1} t_{2}+t_{1}t_{2}t_{3}+t_{1}t_{2}+t_{1}t_{2}+t_{1}t_{1}+t_{1}t_{3}+t_{1}\\
\qquad{} =t_{1}\big( 2t_{1}+2t_{2}+t_{3}+t_{1}^{2}+t_{2}^{2}+t_{1}t_{2}+t_{1}t_{3}+t_{2}t_{3}+1\big) .
\end{gather*}
According to a conjecture of M\'{e}sz\'{a}ros, the result of this substitution is indeed independent of the choices made during the game (as long as $p$ is fixed).
\end{Example}

Why would one play a game like this? The reduction rule $\mathfrak{m}\mapsto\frac{x_{i,k} ( x_{i,j}+x_{j,k}+\beta ) }{x_{i,j}x_{j,k}}\mathfrak{m}$ (this is a~particular case of our above rule, when $\alpha$ is set to $0$) has appeared in Karola M\'{e}sz\'{a}ros's study~\cite{Meszar09} of the abelianization of Anatol Kirillov's quasi-classical Yang--Baxter algebra (see, e.g.,~\cite{Kirill16} for a recent survey of the latter and its many variants); it has a long prehistory (some of which is surveyed in Section~\ref{subsect.arnold} below), starting with Vladimir Arnold's 1971 work~\cite{Arnold71} on the braid arrangement. To define this abelianization\footnote{The notations used in this Introduction are meant to be provisional. In the rest of this paper, we shall work with different notations (and in a more general setting), which will be introduced in Section~\ref{sect.states}.}, we let $\beta$ be an indeterminate (unlike in Example~\ref{exam.intro}, where it was an element of~$\mathbb{Q}$). Furthermore, fix a~positive integer $n$. The abelianization of the ($n$-th) quasi-classical Yang--Baxter algebra is the commutative $\mathbb{Q}[\beta] $-algebra~$\mathcal{S}(A_{n}) $ with
\begin{gather*}
\text{generators}\qquad x_{i,j}\quad \text{for all} \quad 1\leq i<j\leq n\qquad\text{and}\\
\text{relations}\qquad x_{i,j}x_{j,k}=x_{i,k} (x_{i,j}+x_{j,k}+\beta) \quad \text{for all} \quad 1\leq i<j<k\leq n.
\end{gather*}
A natural question is to find an explicit basis of $\mathcal{S} (A_{n}) $ (as a $\mathbb{Q}$-vector space, or, if possible, as a~$\mathbb{Q}[\beta] $-module). One might try constructing such a~basis using a reduction algorithm (or ``straightening law'') that takes any element of $\mathcal{S} ( A_{n}) $ (written as any polynomial in the generators $x_{i,j}$) and rewrites it in a~``normal form''. The most obvious way one could try to construct such a reduction algorithm is by repeatedly rewriting products of the form $x_{i,j}x_{j,k}$ (with $1\leq i<j<k\leq n$) as $x_{i,k} ( x_{i,j}+x_{j,k}+\beta ) $, until this is no longer possible. This is precisely the game that we played in Example~\ref{exam.intro} (with the only difference that $\beta$ is now an indeterminate, not a number). Unfortunately, the result of the game turns out to depend on the choices made while playing it; consequently, the ``normal form'' it constructs is not literally a normal form, and instead of a basis of $\mathcal{S} ( A_{n}) $ we only obtain a spanning set.\footnote{Surprisingly, a similar reduction algorithm \textit{does} work for the (non-abelianized) quasi-classical Yang--Baxter algebra itself. This is one of M\'{e}sz\'{a}ros's results \cite[Theorem 30]{Meszar09}.}

Nevertheless, the result of the game is not meaningless. The idea to substitute $t_{i}$ for $x_{i,j}$ (in the result, not in the original polynomial!) seems to have appeared in work of Postnikov, Stanley and M\'{e}sz\'{a}ros; some concrete formulas (for specific values of the initial polynomial and specific values of $\beta$) appear in \cite[Exercise~A22]{Stanle15} (resulting in Catalan and Narayana numbers). Recent work on Grothendieck polynomials by Anatol Kirillov (see \cite[Section~4]{Kirill13} and \cite{Kirill16}) and by Laura Escobar and Karola M\'{e}sz\'{a}ros \cite[Section~5]{EscMes15} has again brought up the notion of substituting $t_{i}$ for $x_{i,j}$ in the polynomial obtained at the end of the game. This has led M\'{e}sz\'{a}ros to the conjecture that, after this substitution, the resulting polynomial no longer depends on the choices made during the game. She has proven this conjecture for a certain class of polynomials (those corresponding to ``noncrossing trees'').

The main purpose of this paper is to establish M\'{e}sz\'{a}ros's conjecture in the general case. We shall, in fact, work in greater generality than all previously published sources. First, instead of the relation $x_{i,j}x_{j,k}=x_{i,k} ( x_{i,j}+x_{j,k}+\beta ) $, we shall consider the ``deformed'' relation $x_{i,j}x_{j,k}=x_{i,k} ( x_{i,j}+x_{j,k}+\beta ) +\alpha$; the idea of this deformation again goes back to the work of Anatol Kirillov (see, e.g., \cite[Definition~5.1(1)]{Kirill16} for a noncommutative variant of the quotient ring~$\mathcal{X}/\mathcal{J}$, which he calls the ``associative quasi-classical Yang--Baxter algebra of weight $(\alpha,\beta)$''). Instead of requiring $\beta$ to be either a rational number (as in Example~\ref{exam.intro}) or an indeterminate over~$\mathbb{Q}$ (as in the definition of $\mathcal{S}(A_{n})$), we shall let $\beta$ be any element of the ground ring, which in turn will be an arbitrary commutative ring~$\mathbf{k}$. Rather than working in an algebra like $\mathcal{S}(A_{n}) $, we shall work in the polynomial ring $\mathcal{X}=\mathbf{k}[ x_{i,j}\,|\, 1\leq i<j\leq n] $, and study the ideal $\mathcal{J}$ generated by all elements of the form $x_{i,j}x_{j,k}-x_{i,k} ( x_{i,j}+x_{j,k}+\beta ) -\alpha$ for $1\leq i<j<k\leq n$. Instead of focussing on the reduction algorithm, we shall generally study polynomials in $\mathcal{X}$ that are congruent to each other modulo the ideal~$\mathcal{J}$. A monomial in $\mathcal{X}$ will be called ``pathless'' if it is not divisible by
any monomial of the form $x_{i,j}x_{j,k}$ with $i<j<k$. A polynomial in $\mathcal{X}$ will be called ``pathless'' if all monomials appearing in it are pathless. Thus, ``pathless'' polynomials are precisely the polynomials
$p\in\mathcal{X}$ for which the game in Example~\ref{exam.intro} would end immediately if started at~$p$.

Our main result (Theorem~\ref{thm.t-red.unique}) will show that if $p\in\mathcal{X}$ is a polynomial, and if $q\in\mathcal{X}$ is a pathless polynomial congruent to $p$ modulo $\mathcal{J}$, then the image of $q$ under
the substitution $x_{i,j}\mapsto t_{i}$ does not depend on $q$ (but only on $\alpha$, $\beta$ and $p$). This, in particular, yields M\'{e}sz\'{a}ros's conjecture; but it is a stronger result, because it does not require that~$q$ is obtained from $p$ by playing the game from Example~\ref{exam.intro} (all we ask for is that $q$ be pathless and congruent to~$p$ modulo~$\mathcal{J}$), and of course because of the more general setting.

After the proof of Theorem~\ref{thm.t-red.unique}, we shall rewrite the definition of $\mathcal{J}$ (and of $\mathcal{X}$) in a more symmetric form (Section~\ref{sect.symmetry}). Then, we shall also answer the (easier)
question of finding a~basis for the quotient ring $\mathcal{X}/\mathcal{J}$ (Proposition~\ref{prop.forkless.basis}). This basis will be obtained using an explicit Gr\"{o}bner basis of the ideal $\mathcal{J}$.

We shall close with further considerations, open questions and connections to previous research.

A recent preprint by M\'{e}sz\'{a}ros and St.~Dizier \cite{MesDiz17} proves a fact \cite[Theorem~A]{MesDiz17} which, translated into our language, confirms the conjecture stated in Example~\ref{exam.intro} at least in the case when $\alpha=0$ and the game is started with a monomial~$p$. This might provide a~different route to some of our results. (The arguments in~\cite{MesDiz17} are of combinatorial nature, involving flows on graphs, and so is the language used in~\cite{MesDiz17}; in particular, monomials are encoded by graphs.)

\subsection{Remark on alternative versions}

This paper also has a detailed version~\cite{verlong}, which includes some proofs that have been omitted from the present version (mostly straightforward computations and basic properties of Gr\"{o}bner bases).

In a previous version (\href{https://arxiv.org/abs/1704.00839v2}{arXiv:1704.00839v2}) of this paper, a weaker version of the main result was proven (which corresponds to the case $\alpha=0$ in our notations). The proof used a somewhat different construction (involving formal power series instead of Laurent series, and a different map~$A$).

\section{Definitions and results}\label{sect.states}

Let us now start from scratch, and set the stage for the main result.

\begin{Definition}Let $\mathbb{N}=\{ 0,1,2,\ldots\} $. Let $[m] $ be the set $\{ 1,2,\ldots,m\} $ for each $m\in\mathbb{N}$. Let $\mathbf{k}$ be a commutative ring. (We fix $\mathbf{k}$ throughout this paper.) Fix two elements $\beta$ and $\alpha$ of~$\mathbf{k}$.

The word ``monomial'' shall always mean an element of a free abelian monoid (written multiplicatively). For example, the monomials in two indeterminates $x$ and $y$ are the elements of the form $x^{i}y^{j}$ with $(i,j) \in\mathbb{N}^{2}$. Thus, monomials do not include coefficients (and are not bound to a specific base ring).
\end{Definition}

\begin{Definition}Fix a positive integer $n$. Let $\mathcal{X}$ be the polynomial ring
\begin{gather*}
\mathbf{k}\big[ x_{i,j}\,|\, (i,j) \in[n] ^{2}\text{ satisfying }i<j\big] .
\end{gather*}
This is a polynomial ring in $n(n-1) /2$ indeterminates $x_{i,j}$ over $\mathbf{k}$.

We shall use the notation $\mathfrak{M}$ for the set of all monomials in these indeterminates $x_{i,j}$. Notice that $\mathfrak{M}$ is an abelian monoid under multiplication.
\end{Definition}

\begin{Definition} A monomial $\mathfrak{m}\in\mathfrak{M}$ is said to be \textit{pathless} if there exists no triple $(i,j,k) \in[n] ^{3}$ satisfying $i<j<k$ and $x_{i,j}x_{j,k}\,|\,\mathfrak{m}$ (as monomials).

A polynomial $p\in\mathcal{X}$ is said to be \textit{pathless} if it is a~$\mathbf{k}$-linear combination of pathless monomials.
\end{Definition}

\begin{Definition} Let $\mathcal{J}$ be the ideal of $\mathcal{X}$ generated by all elements of the form $x_{i,j}x_{j,k}-x_{i,k} ( x_{i,j}+x_{j,k}+\beta ) -\alpha$ for $(i,j,k) \in[n] ^{3}$ satisfying $i<j<k$.
\end{Definition}

The following fact is easy to check:

\begin{Proposition}\label{prop.path-red.span}Let $p\in\mathcal{X}$. Then, there exists a pathless polynomial $q\in\mathcal{X}$ such that $p\equiv q$ $\operatorname{mod} \mathcal{J}$.
\end{Proposition}

In general, this $q$ is not unique.\footnote{For instance, if $\mathbf{k}=\mathbb{Z}$, $\beta=1$, $\alpha=0$ and $n=4$, then
\begin{gather*}
q_{1} =x_{1,2}x_{1,3}x_{1,4}+x_{1,2}x_{1,4}+x_{1,2}x_{1,4}x_{3,4} +x_{1,3}x_{1,4}+x_{1,3}x_{1,4}x_{2,4}+x_{1,3}x_{2,3}x_{2,4}\\
\hphantom{q_{1} =}{} +x_{1,3}x_{2,4}+x_{1,4}+x_{1,4}x_{2,4}+x_{1,4}x_{2,4}x_{3,4}+x_{1,4}x_{3,4}%
\end{gather*}
and%
\begin{gather*}
q_{2} =x_{1,2}x_{1,3}x_{1,4}+x_{1,2}x_{1,4}+x_{1,2}x_{1,4}x_{3,4}+x_{1,3}x_{1,4}+x_{1,3}x_{1,4}x_{2,3}\\
\hphantom{q_{2} =}{} +x_{1,4}+x_{1,4}x_{2,3}+x_{1,4}x_{2,3}x_{2,4}+x_{1,4}x_{2,4}+x_{1,4}x_{2,4}x_{3,4}+x_{1,4}x_{3,4}
\end{gather*}
are two pathless polynomials $q\in\mathcal{X}$ satisfying $x_{1,2}x_{2,3}x_{3,4}\equiv q\operatorname{mod}\mathcal{J}$, but they are not identical.}

We shall roughly outline a proof of Proposition \ref{prop.path-red.span} now; a detailed writeup of this proof can be found in the detailed version~\cite{verlong} of this paper.

\begin{proof}[Proof of Proposition \ref{prop.path-red.span} (sketched).]The \textit{weight} of a monomial $\prod_{\substack{(i,j) \in[n] ^{2};\\i<j}}x_{i,j}^{a_{i,j}}\in\mathfrak{M}$ shall mean the nonnegative integer $\sum_{\substack{(i,j) \in[n] ^{2};\\i<j}}a_{i,j}(n-j+i) $. If we have a monomial $\mathfrak{m}\in\mathfrak{M}$ that is not pathless, then we can find a triple $ (i,j,k) \in[n] ^{3}$ satisfying $i<j<k$ and $x_{i,j}x_{j,k}\,|\,\mathfrak{m}$; then, we can replace $\mathfrak{m}$ by a polynomial $\widetilde{\mathfrak{m}}=\mathfrak{m}\cdot\frac{x_{i,k}(x_{i,j}+x_{j,k}+\beta) +\alpha}{x_{i,j}x_{j,k}}$, which is congruent to~$\mathfrak{m}$ modulo $\mathcal{J}$ but has the property that all monomials appearing in it have a smaller weight than~$\mathfrak{m}$. This gives rise to a recursive algorithm\footnote{Or ``straightening law'', as algorithms of this kind are commonly called in algebraic combinatorics.} for reducing a polynomial modulo the ideal~$\mathcal{J}$. The procedure will necessarily terminate (although its result might depend on the order of operation); the polynomial resulting at its end will be pathless.
\end{proof}

The ideal $\mathcal{J}$ is relevant to the so-called \textit{subdivision algebra of root polytopes} (denoted by $\mathcal{S}(\beta) $ in \cite[Section~5]{EscMes15} and $\mathcal{S}(A_{n}) $ in \cite[Section~1]{Meszar09}). Namely, this latter algebra is defined as the quotient~$\mathcal{X}/\mathcal{J}$ for a certain choice of~$\mathbf{k}$,~$\beta$ and~$\alpha$ (namely, for the choice where~$\mathbf{k}$ is a univariate polynomial ring over $\mathbb{Q}$, where $\beta$ is the indeterminate in~$\mathbf{k}$, and where $\alpha=0$). This algebra was first introduced by M\'{e}sz\'{a}ros in~\cite{Meszar09} as the abelianization of Anatol Kirillov's quasi-classical
Yang--Baxter algebra.

In \cite[Section~5 and Appendix~A]{EscMes15}, Escobar and M\'{e}sz\'{a}ros (motivated by computations of Grothendieck polynomials) consider the result of substituting~$t_{i}$ for each variable~$x_{i,j}$ in a polynomial $f\in\mathcal{X}$. In our language, this leads to the following definition:

\begin{Definition}Let $\mathcal{T}^{\prime}$ be the polynomial ring $\mathbf{k} [ t_{1},t_{2},\ldots,t_{n-1}] $. We define a $\mathbf{k}$-algebra homomorphism $D\colon \mathcal{X}\rightarrow\mathcal{T}^{\prime}$ by
\begin{gather*}
D ( x_{i,j} ) =t_{i}\qquad \text{for every } (i,j) \in[n] ^{2}\text{ satisfying }i<j.
\end{gather*}
\end{Definition}

The goal of this paper is to prove the following fact, which (in a less general setting) was conjectured by Karola M\'{e}sz\'{a}ros in a 2015 talk at MIT:

\begin{Theorem}\label{thm.t-red.unique}Let $p\in\mathcal{X}$. Consider any pathless polynomial $q\in\mathcal{X}$ such that $p\equiv q\operatorname{mod} \mathcal{J}$. Then, $D(q) $ does not depend on the choice of~$q$ $($but merely on the choice of $\alpha$, $\beta$ and~$p)$.
\end{Theorem}

It is not generally true that $D(q) =D(p) $; thus, Theorem~\ref{thm.t-red.unique} does not follow from a simple ``invariant''.

\section{The proof}\label{sect.proof}

\subsection{Preliminaries}

The proof of Theorem~\ref{thm.t-red.unique} will occupy most of this paper. It proceeds in several steps. First, we shall define four $\mathbf{k}$-algebras $\mathcal{Q}$, $\mathcal{T}^{\prime} [ [w] ] $, $\mathcal{T}$ and $\mathcal{T}[[w]] $ (with $\mathcal{T}^{\prime}$ being a subalgebra of $\mathcal{T}$) and three $\mathbf{k}$-linear maps $A$, $B$ and $E$ (with $A$ and $E$ being $\mathbf{k}$-algebra homomorphisms) forming a diagram
\begin{gather*}
\xymatrix{\calX\ar[r]^{A} \ar[dr]_{D} & \calQ\ar[r]^-{B} & \calT [ [w]] \\
& \calT^\prime\ar[r]_-{E} & \calT^\prime[[w]] \arinj[u]
}
\end{gather*}
(where the vertical arrow is a canonical injection) that is \textit{not} commutative. We shall eventually show that:
\begin{itemize}\itemsep=0pt
\item (Proposition \ref{prop.A.killsJ} below) the homomorphism $A$ annihilates the ideal $\mathcal{J}$,

\item (Proposition \ref{prop.E.inj} below) the homomorphism $E$ is injective, and

\item (Corollary \ref{cor.pathless.D} below) each pathless polynomial $q$ satisfies $( E\circ D) (q) =( B\circ A ) (q) $ (the equation makes sense since $\mathcal{T}^{\prime}[[w]] \subseteq\mathcal{T}[[w]]$).
\end{itemize}

These three facts will allow us to prove Theorem~\ref{thm.t-red.unique}. Indeed, the first and the third will imply that each pathless polynomial in~$\mathcal{J}$ is annihilated by~$E\circ D$; because of the second, this will
show that it is also annihilated by~$D$; and from here, Theorem~\ref{thm.t-red.unique} will easily follow.

\subsection[The algebra $\mathcal{Q}$ of Laurent series]{The algebra $\boldsymbol{\mathcal{Q}}$ of Laurent series}

Let us begin by defining the notion of (formal) Laurent series in $n$ indeterminates $r_{1},r_{2},\ldots,r_{n}$. This is somewhat slippery terrain, and it is easy to accidentally get a non-working definition (e.g., a notion of ``Laurent series'' not closed under multiplication, or not allowing multiplication at all), but there are also several different working definitions (see, e.g., \cite{MonKau13} for a~systematic treatment revealing many degrees of freedom). The definition we shall give here has been tailored to make our constructions work.

We begin by defining a $\mathbf{k}$-module $\mathcal{Q}^{\pm}$ of ``two-sided infinite formal power series over $\mathbf{k}$''; this is not going to be a ring:

\begin{Definition} \label{def.Qbig}Consider $n$ distinct symbols $r_{1},r_{2},\ldots,r_{n}$. Let $\mathfrak{R}$ denote the free abelian group on these $n$ symbols, written multiplicatively. (That is, $\mathfrak{R}$ is the free $\mathbb{Z}$-module on $n$ generators $r_{1},r_{2},\ldots,r_{n}$, but with the addition renamed as multiplication.) The elements of $\mathfrak{R}$ thus have the form $r_{1}^{a_{1}}r_{2}^{a_{2}}\cdots r_{n}^{a_{n}}$ for $( a_{1},a_{2},\ldots,a_{n}) \in\mathbb{Z}^{n}$; we shall refer to such elements as \textit{Laurent monomials} in the symbols $r_{1},r_{2},\ldots,r_{n}$.

Informally, we let $\mathcal{Q}^{\pm}$ denote the $\mathbf{k}$-module of all ``infinite $\mathbf{k}$-linear combinations'' of Laurent monomials. Formally speaking, we define $\mathcal{Q}^{\pm}$ as the direct product $\prod_{\mathfrak{r}\in\mathfrak{R}}\mathbf{k}$ of copies of~$\mathbf{k}$ indexed by Laurent monomials. We want to write each element $ ( \lambda_{\mathfrak{r}}) _{\mathfrak{r}\in\mathfrak{R}}\in \prod_{\mathfrak{r}\in\mathfrak{R}}\mathbf{k}$ of this direct product as the formal $\mathbf{k}$-linear combination $\sum_{\mathfrak{r}\in\mathfrak{R}}\lambda_{\mathfrak{r}}\mathfrak{r}$; in order for this to work, we make
several further conventions: First, we identify each Laurent monomial $\mathfrak{s}\in\mathfrak{R}$ with the element $( \delta_{\mathfrak{s},\mathfrak{r}}) _{\mathfrak{r}\in\mathfrak{R}}$ of $\mathcal{Q}^{\pm}$ (where $\delta_{\mathfrak{s},\mathfrak{r}}$ is the Kronecker delta). Second, we equip the $\mathbf{k}$-module $\mathcal{Q}^{\pm}$ with a topology: namely, the product topology, defined by recalling that it is a direct product
$\prod_{\mathfrak{r}\in\mathfrak{R}}\mathbf{k}$ of copies of $\mathbf{k}$ (each of which is equipped with the discrete topology). Having made these conventions, we can easily verify that each element $( \lambda_{\mathfrak{r}}) _{\mathfrak{r}\in\mathfrak{R}}$ of $\prod_{\mathfrak{r}\in\mathfrak{R}}\mathbf{k}=\mathcal{Q}^{\pm}$ is indeed identical with the infinite sum $\sum_{\mathfrak{r}\in\mathfrak{R}}\lambda_{\mathfrak{r}}\mathfrak{r}$ (which makes sense because of the topology on $\mathcal{Q}^{\pm}$). As usual, if $f= ( \lambda_{\mathfrak{r}} ) _{\mathfrak{r}\in\mathfrak{R}}$ is an element of~$\mathcal{Q}^{\pm}$, then $\lambda_{\mathfrak{r}}$ (for a given $\mathfrak{r}\in\mathfrak{R}$) will be called the \textit{coefficient} of $\mathfrak{r}$ in $f$ and denoted by $ [ \mathfrak{r} ] f$.
\end{Definition}

As we know, the Laurent monomials in $\mathfrak{R}$ have the form $r_{1}^{a_{1}}r_{2}^{a_{2}}\cdots r_{n}^{a_{n}}$ for $( a_{1},a_{2},\ldots,a_{n}) \in\mathbb{Z}^{n}$; thus, sums of the form $\sum_{\mathfrak{r}\in\mathfrak{R}}\lambda_{\mathfrak{r}}\mathfrak{r}$ can also be rewritten in the form
\begin{gather*}
\sum_{( a_{1},a_{2},\ldots,a_{n})\in\mathbb{Z}^{n}}\lambda_{a_{1},a_{2},\ldots,a_{n}}r_{1}^{a_{1}}r_{2}^{a_{2}}\cdots r_{n}^{a_{n}} ;
\end{gather*}
this is the usual way in which elements of $\mathcal{Q}^{\pm}$ are written.

For example, for $n=1$, an element of $\mathcal{Q}^{\pm}$ will have the form $\sum_{a\in\mathbb{Z}}\lambda_{a}r_{1}^{a}$ for some family $ (\lambda_{a}) _{a\in\mathbb{Z}}$ of elements of~$\mathbf{k}$. Already in this simple situation, we see that $\mathcal{Q}^{\pm}$ is not a ring (or, at least, the usual recipe for multiplying power series does not work in $\mathcal{Q}^{\pm}$): multiplying $\sum_{a\in\mathbb{Z}}r_{1}^{a}$ with itself
would result in
\begin{gather*}
\left( \sum_{a\in\mathbb{Z}}r_{1}^{a}\right) \left( \sum_{a\in\mathbb{Z}}r_{1}^{a}\right) =\sum_{\left( a,b\right) \in\mathbb{Z}^{2}}r_{1}^{a+b},
\end{gather*}
which is not a convergent sum in any reasonable topology (it contains each Laurent monomial infinitely many times). We shall define Laurent series as a~subring of $\mathcal{Q}^{\pm}$:

\begin{Definition}\label{def.laurent}\quad
\begin{enumerate}\itemsep=0pt
\item[(a)] If $d$ is an integer and $\mathfrak{r} \in\mathfrak{R}$ is a Laurent monomial, then we say that $\mathfrak{r}$ \textit{lives above }$d$ if and only if $\mathfrak{r}=r_{1}^{a_{1}}r_{2}^{a_{2}}\cdots r_{n}^{a_{n}}$ for some $( a_{1},a_{2},\ldots,a_{n}) \in\{ d,d+1,d+2,\ldots\} ^{n}$.

\item[(b)] If $d$ is an integer and $f$ is an element of $\mathcal{Q}^{\pm}$, then we say that $f$ is \textit{supported above }$d$ if and only if every $\left( a_{1},a_{2},\ldots,a_{n}\right) \in\mathbb{Z}^{n}\setminus\{d,d+1,d+2,\ldots \} ^{n}$ satisfies $\left[ r_{1}^{a_{1}}r_{2}^{a_{2} }\cdots r_{n}^{a_{n}}\right] f=0$. In other words, $f$ is supported above $d$ if and only if $f$ is an infinite $\mathbf{k}$-linear combination of Laurent monomials that live above $d$.

\item[(c)] An element $f\in\mathcal{Q}^{\pm}$ is said to be a \textit{Laurent series} if and only if there exists some $d\in\mathbb{Z}$ such that $f$ is supported above $d$.

\item[(d)] We let $\mathbf{k}\left( \left( r_{1},r_{2},\ldots
,r_{n}\right) \right) $ denote the $\mathbf{k}$-submodule of $\mathcal{Q}%
^{\pm}$ consisting of all Laurent series.

\item[(e)] A multiplication can be defined on $\mathbf{k}(( r_{1},r_{2},\ldots,r_{n})) $ by extending the multiplication in the group $\mathfrak{R}$ (in such a way that the resulting map is bilinear and continuous). Explicitly, this means that if $f=\sum_{\mathfrak{r} \in\mathfrak{R}}\lambda_{\mathfrak{r}}\mathfrak{r}$ and $g=\sum_{\mathfrak{r} \in\mathfrak{R}}\mu_{\mathfrak{r}}\mathfrak{r}$ are two Laurent series, then their product~$fg$ is defined as the Laurent series%
\begin{gather*}
\left( \sum_{\mathfrak{r}\in\mathfrak{R}}\lambda_{\mathfrak{r}} \mathfrak{r}\right) \left( \sum_{\mathfrak{r}\in\mathfrak{R}}\mu _{\mathfrak{r}}\mathfrak{r}\right) =\sum_{\mathfrak{u}\in\mathfrak{R}} \sum_{\mathfrak{v}\in\mathfrak{R}}\lambda_{\mathfrak{u}}\mu_{\mathfrak{v} }\mathfrak{uv}=\sum_{\mathfrak{r}\in\mathfrak{R}}\left( \sum _{\substack{\left( \mathfrak{u},\mathfrak{v}\right) \in\mathfrak{R}^{2};\\\mathfrak{uv}=\mathfrak{r}}}\lambda_{\mathfrak{u}}\mu_{\mathfrak{v}}\right) \mathfrak{r}.
\end{gather*}
The inner sum $\sum_{\substack{ ( \mathfrak{u},\mathfrak{v} ) \in\mathfrak{R}^{2};\\\mathfrak{uv}=\mathfrak{r}}}\lambda_{\mathfrak{u}} \mu_{\mathfrak{v}}$ here is well-defined, because all but finitely many of its addends are zero. (In fact, if $f$ is supported above~$d$, and if~$g$ is supported above $e$, then (for each given $\mathfrak{r}\in\mathfrak{R}$) there are only finitely many pairs $ ( \mathfrak{u},\mathfrak{v} ) \in\mathfrak{R}^{2}$ such that $\mathfrak{uv}=\mathfrak{r}$ and $\mathfrak{u}$ lives above $d$ and $\mathfrak{v}$ lives above $e$; but these are the only pairs that can contribute nonzero addends to the sum $\sum_{\substack{(\mathfrak{u},\mathfrak{v}) \in\mathfrak{R}^{2};\\\mathfrak{uv}=\mathfrak{r}}}\lambda_{\mathfrak{u}}\mu_{\mathfrak{v}}$.)

Thus, $\mathbf{k} ( ( r_{1},r_{2},\ldots,r_{n}) ) $ becomes a $\mathbf{k}$-algebra with unity $1=r_{1}^{0}r_{2}^{0}\cdots r_{n}^{0}$. We denote this $\mathbf{k}$-algebra by $\mathcal{Q}$. Note that $\mathcal{Q}$ is a topological $\mathbf{k}$-algebra; its topology is inherited from $\mathcal{Q}^{\pm}$.

\item[(f)] An element $f\in\mathcal{Q}^{\pm}$ is said to be a \textit{formal power series} if and only if $f$ is supported above~$0$.

\item[(g)] We let $\mathbf{k}[ [ r_{1},r_{2},\ldots,r_{n}] ] $ denote the $\mathbf{k}$-submodule of $\mathcal{Q}^{\pm}$ consisting of all formal power series. Thus, $\mathbf{k} [ [ r_{1},r_{2},\ldots,r_{n}] ] \subseteq\mathcal{Q}\subseteq\mathcal{Q}^{\pm}$.
\end{enumerate}
\end{Definition}

We now define certain Laurent monomials $q_{1},q_{2},\ldots,q_{n}$ that we shall often use:

\begin{Definition} \label{def.Q'}For each $i\in[n] $, we define a Laurent monomial $q_{i}$ in the indeterminates $r_{1},r_{2},\ldots,r_{n}$ by $q_{i}=r_{i}r_{i+1}\cdots r_{n}$. Notice that this $q_{i}$ is an actual monomial, not only a Laurent monomial.

Notice that each Laurent monomial in $\mathfrak{R}$ belongs to $\mathcal{Q}$. Each of the elements $q_{1},q_{2},\ldots,q_{n}$ of~$\mathcal{Q}$ is a Laurent monomial, and thus has an inverse (in $\mathfrak{R}$ and thus also in
$\mathcal{Q}$). Hence, it makes sense to speak of quotients such as $q_{i}/q_{j}$ for $1\leq i\leq j\leq n$. Explicitly, $q_{i}/q_{j}=r_{i} r_{i+1}\cdots r_{j-1}$ whenever $1\leq i\leq j\leq n$. Thus, for any $i\in[n] $ and $j\in[n] $ satisfying $i<j$, the difference $1-q_{i}/q_{j}=1-r_{i}r_{i+1}\cdots r_{j-1}$ is a formal power series in $\mathbf{k} [ [ r_{1},r_{2},\ldots,r_{n} ] ] $ having constant term~$1$; it is therefore invertible in $\mathbf{k}[[ r_{1},r_{2},\ldots,r_{n}]] $.
\end{Definition}

It is easy to see that
\begin{gather}
q_{1}^{a_{1}}q_{2}^{a_{2}}\cdots q_{n}^{a_{n}}=r_{1}^{a_{1}}r_{2}^{a_{1}+a_{2}}r_{3}^{a_{1}+a_{2}+a_{3}}\cdots r_{n}^{a_{1}+a_{2}+\cdots+a_{n}}\label{eq.Q'.reindex}
\end{gather}
for all $ ( a_{1},a_{2},\ldots,a_{n} ) \in\mathbb{Z}^{n}$. Also,
\begin{gather*}
r_{1}^{b_{1}}r_{2}^{b_{2}}\cdots r_{n}^{b_{n}}=q_{1}^{b_{1}}q_{2}^{b_{2}-b_{1}}q_{3}^{b_{3}-b_{2}}\cdots q_{n}^{b_{n}-b_{n-1}} %\label{eq.Q'.reindex2}
\end{gather*}
for all $( b_{1},b_{2},\ldots,b_{n}) \in\mathbb{Z}^{n}$. Thus, each Laurent monomial $\mathfrak{r}\in\mathfrak{R}$ can be written uniquely in the form $q_{1}^{a_{1}}q_{2}^{a_{2}}\cdots q_{n}^{a_{n}}$ with $(a_{1},a_{2},\ldots,a_{n} ) \in\mathbb{Z}^{n}$. Thus, $ \big(q_{1}^{a_{1}}q_{2}^{a_{2}}\cdots q_{n}^{a_{n}}\big) _{( a_{1},a_{2},\ldots,a_{n}) \in\mathbb{Z}^{n}}$ is a topological basis\footnote{The notion of a ``topological basis'' that we are using here has nothing to do with the concept of a~basis of a~topology (also known as ``base''). Instead, it is merely an analogue of the concept of a~basis of a $\mathbf{k}$-module. It is defined as follows:

A \textit{topological basis} of a topological $\mathbf{k}$-module $\mathcal{M}$ means a family $ ( m_{s} ) _{s\in\mathfrak{S}}\in\mathcal{M}^{\mathfrak{S}}$ with the following two properties:
\begin{itemize}\itemsep=0pt
\item For each family $\left( \lambda_{s}\right) _{s\in\mathfrak{S}} \in\mathbf{k}^{\mathfrak{S}}$, the sum $\sum_{s\in\mathfrak{S}}\lambda _{s}m_{s}$ converges with respect to the topology on $\mathcal{M}$. (Such a~sum is called an \textit{infinite} $\mathbf{k}$\textit{-linear combination} of the family $ ( m_{s} ) _{s\in\mathfrak{S}}$.)
\item Each element of $\mathcal{M}$ can be uniquely represented in the form $\sum_{s\in\mathfrak{S}}\lambda_{s}m_{s}$ for some family $( \lambda _{s}) _{s\in\mathfrak{S}}\in\mathbf{k}^{\mathfrak{S}}$.
\end{itemize}

For example, $ ( r_{1}^{b_{1}}r_{2}^{b_{2}}\cdots r_{n}^{b_{n}} ) _{( b_{1},b_{2},\ldots,b_{n}) \in\mathbb{N}^{n}}$ is a topological basis of the topological $\mathbf{k}$-module $\mathbf{k} [ [ r_{1},r_{2},\ldots,r_{n} ] ] $, because each power series can be uniquely represented as an infinite $\mathbf{k}$-linear combination of all the monomials.} of the $\mathbf{k}$-module $\mathcal{Q}^{\pm}$.

\subsection[The algebra homomorphism $A\colon \mathcal{X}\rightarrow\mathcal{Q}$]{The algebra homomorphism $\boldsymbol{A\colon \mathcal{X}\rightarrow\mathcal{Q}}$}

\begin{Definition}\label{def.A}Define a $\mathbf{k}$-algebra homomorphism $A\colon \mathcal{X}\rightarrow\mathcal{Q}$ by
\begin{gather*}
A ( x_{i,j} ) =-\frac{q_{i}+\beta+\alpha/q_{j}}{1-q_{i}/q_{j}}\qquad \text{for all }(i,j) \in[n] ^{2}\text{ satisfying }i<j.
\end{gather*}
Notice that this is well-defined, since all denominators appearing here are invertible (indeed, $q_{j}$~is an invertible Laurent monomial in~$\mathfrak{R}$, and $1-q_{i}/q_{j}$ is an invertible formal power series in
$\mathbf{k} [[ r_{1},r_{2},\ldots,r_{n} ] ] $).
\end{Definition}

\begin{Proposition}\label{prop.A.killsJ}We have $A ( \mathcal{J} ) =0$.
\end{Proposition}

\begin{proof}The ideal $\mathcal{J}$ of $\mathcal{X}$ is generated by all elements of the form $x_{i,j}x_{j,k}-x_{i,k} ( x_{i,j}+x_{j,k}+\beta ) -\alpha$ for all triples
$(i,j,k) \in[n] ^{3}$ satisfying $i<j<k$. Thus, it suffices to show that $A ( x_{i,j}x_{j,k}-x_{i,k} ( x_{i,j}+x_{j,k}+\beta ) -\alpha ) =0$ for all such triples. But this is a~straightforward computation (see~\cite{verlong} for the details).
\end{proof}

\subsection[The algebras $\mathcal{T}$ and $\protect{\mathcal{T}[[w]]}$ of power series]{The algebras $\boldsymbol{\mathcal{T}}$ and $\boldsymbol{\mathcal{T}[[w]]}$ of power series}

\begin{Definition}\quad\begin{enumerate}\itemsep=0pt
\item[(a)] Let $\mathcal{T}$ be the topological $\mathbf{k}$-algebra $\mathbf{k} [ [ t_{1},t_{2},\ldots,t_{n} ] ] $. This is the ring of formal power series in the $n$ indeterminates $t_{1},t_{2},\ldots,t_{n}$ over $\mathbf{k}$.

The topology on $\mathcal{T}$ shall be the usual one (i.e., the one defined similarly to the one on~$\mathcal{Q}^{\pm}$).

\item[(b)] We shall regard the canonical injections%
\begin{gather*}
\mathcal{T}^{\prime}=\mathbf{k} [ t_{1},t_{2},\ldots,t_{n-1} ] \hookrightarrow\mathbf{k}[ t_{1},t_{2},\ldots,t_{n}]\hookrightarrow\mathbf{k} [ [ t_{1},t_{2},\ldots,t_{n} ] ] =\mathcal{T}
\end{gather*}
as inclusions. Thus, $\mathcal{T}^{\prime}$ becomes a~$\mathbf{k}$-subalgebra of~$\mathcal{T}$. Hence, $D\colon \mathcal{X}\rightarrow\mathcal{T}^{\prime}$ becomes a~$\mathbf{k}$-algebra homomorphism $\mathcal{X}\rightarrow
\mathcal{T}$.

\item[(c)] We consider the $\mathbf{k}$-algebras $\mathcal{T} [ [ w] ] $ and $\mathcal{T}^{\prime}[ [w]] $. These are the $\mathbf{k}$-algebras of formal power series in a~(new) indeterminate $w$ over $\mathcal{T}$ and over $\mathcal{T}^{\prime}$, respectively. We endow the $\mathbf{k}$-algebra $\mathcal{T}[ [w]] $ with a topology defined as the product topology, where~$\mathcal{T}[[w]] $ is identified with a direct product of infinitely many copies of $\mathcal{T}$ (each of which is equipped with the topology we previously defined).
\end{enumerate}
\end{Definition}

\subsection[The continuous $\mathbf{k}$-linear map $\protect{B\colon \mathcal{Q}\rightarrow\mathcal{T}[[w]]}$]{The continuous $\boldsymbol{\mathbf{k}}$-linear map $\boldsymbol{B\colon \mathcal{Q}\rightarrow\mathcal{T}[[w]]}$}

We have $\mathcal{T}=\mathbf{k}[ [ t_{1},t_{2},\ldots,t_{n}] ] $. Thus, $\mathcal{T}[ [w]] $ can be regarded as the ring of formal power series in the $n+1$ indeterminates $t_{1},t_{2},\ldots,t_{n},w$ over~$\mathbf{k}$. (Strictly speaking, this should say that there is a canonical topological $\mathbf{k}$-algebra isomorphism from $\mathcal{T}[[w]]$ to
the latter ring). Let us now show a simple lemma:

\begin{Lemma}\label{lem.B.wd1}Let $\mathfrak{m}$ be a monomial in the indeterminates $t_{1},t_{2},\ldots,t_{n},w$ $($with nonnegative exponents$)$. Then, there exist only finitely many $( a_{1},a_{2},\ldots,a_{n}) \in\mathbb{Z}^{n}$ satisfying
\begin{gather}
\left( \prod_{\substack{i\in[n] ;\\a_{i}>0}}t_{i}^{a_{i}}\right) \left( \prod_{\substack{i\in[n] ;\\a_{i}<0}}w^{-a_{i}}\right) =\mathfrak{m}. \label{eq.lem.B.wd1.eq}
\end{gather}
\end{Lemma}

\begin{proof} Write $\mathfrak{m}$ in the form $\mathfrak{m}=\Big( \prod_{i\in[n] }t_{i}^{b_{i}}\Big) w^{c}$ for some nonnegative integers $b_{1},b_{2},\ldots,b_{n},c$. Let $S$ be the finite set $ \{ -c,-c+1,\ldots,0 \} \cup \{ b_{1} ,b_{2},\ldots,b_{n} \} $. Hence, $S^{n}$ is also a~finite set.

We want to prove that there exist only finitely many $ ( a_{1},a_{2},\ldots,a_{n} ) \in\mathbb{Z}^{n}$ satisfying~(\ref{eq.lem.B.wd1.eq}). We shall show that each such $ ( a_{1},a_{2},\ldots,a_{n} ) $ belongs to the set~$S^{n}$.

Indeed, let $ ( a_{1},a_{2},\ldots,a_{n} ) \in\mathbb{Z}^{n}$ satisfy (\ref{eq.lem.B.wd1.eq}). We must show that $ ( a_{1},a_{2},\ldots,a_{n} ) $ belongs to~$S^{n}$.

Let $j\in[n] $. We want to prove that $a_{j}\in S$.

We know that (\ref{eq.lem.B.wd1.eq}) holds. Thus
\begin{gather*}
\left( \prod_{\substack{i\in[n] ;\\a_{i}>0}}t_{i}^{a_{i}}\right) \left( \prod_{\substack{i\in[n] ;\\a_{i}<0}}w^{-a_{i}}\right) =\mathfrak{m}=\left( \prod_{i\in[n] }t_{i}^{b_{i}}\right) w^{c}.
\end{gather*}
This is an equality between two monomials in the indeterminates $t_{1},t_{2},\ldots,t_{n},w$. Comparing exponents on both sides of this equality, we find that
\begin{gather}
b_{i}=
\begin{cases}
a_{i}, & \text{if }a_{i}>0,\\
0, & \text{otherwise}
\end{cases}
\qquad \text{for each }i\in[n] \label{pf.lem.B.wd1.bi=}
\end{gather}
and
\begin{gather}
c=\sum_{\substack{i\in[n] ;\\a_{i}<0}} ( -a_{i} ) .\label{pf.lem.B.wd1.c=}
\end{gather}

Now, we are in one of the following three cases:
\begin{enumerate}\itemsep=0pt
\item[] \textit{Case 1:} We have $a_{j}<0$.

\item[] \textit{Case 2:} We have $a_{j}=0$.

\item[] \textit{Case 3:} We have $a_{j}>0$.
\end{enumerate}

Let us first consider Case 1. In this case, we have $a_{j}<0$. Thus, $-a_{j}$ is one of the addends in the sum $\sum_{\substack{i\in[n] ;\\a_{i}<0}}(-a_{i}) $. Since this sum is greater or equal to each of its addends (because its addends are positive), we thus obtain $\sum_{\substack{i\in[n] ;\\a_{i}<0}}(-a_{i}) \geq-a_{j}$. Hence, (\ref{pf.lem.B.wd1.c=}) becomes $c=\sum_{\substack{i\in [n] ;\\a_{i}<0}}(-a_{i}) \geq-a_{j}$. In other
words, $a_{j}\geq-c$. Combining this with $a_{j}<0$, we find
\begin{gather*}
a_{j}\in\{ -c,-c+1,\ldots,-1\} \subseteq \{ -c,-c+1,\ldots,0 \} \cup \{ b_{1},b_{2},\ldots,b_{n} \} =S.
\end{gather*}
Thus, $a_{j}\in S$ is proven in Case 1.

Let us now consider Case 2. In this case, we have $a_{j}=0$. Hence,
\begin{gather*}
a_{j}\in \{ -c,-c+1,\ldots,0 \} \subseteq \{ -c,-c+1,\ldots ,0 \} \cup \{ b_{1},b_{2},\ldots,b_{n} \} =S.
\end{gather*}
Thus, $a_{j}\in S$ is proven in Case 2.

Let us finally consider Case 3. In this case, we have $a_{j}>0$. Applying (\ref{pf.lem.B.wd1.bi=}) to $i=j$, we find
\begin{gather*}
b_{j}=
\begin{cases}
a_{j}, & \text{if }a_{j}>0;\\
0, & \text{otherwise}
\end{cases}
=a_{j}\qquad ( \text{since }a_{j}>0 ) ,
\end{gather*}
so that%
\begin{gather*}
a_{j}=b_{j}\in \{ b_{1},b_{2},\ldots,b_{n} \} \subseteq \{ -c,-c+1,\ldots,0\} \cup\{ b_{1},b_{2},\ldots,b_{n}\} =S.
\end{gather*}
Thus, $a_{j}\in S$ is proven in Case 3.

We have now proven $a_{j}\in S$ in all three Cases 1, 2 and 3. Hence, $a_{j}\in S$ always holds.

Forget that we have fixed $j$. We thus have shown that $a_{j}\in S$ for each $j\in[n] $. In other words, $ ( a_{1},a_{2},\ldots,a_{n}) $ belongs to $S^{n}$.

Now, forget that we fixed $ ( a_{1},a_{2},\ldots,a_{n} ) $. We thus have shown that each $ ( a_{1},a_{2},\ldots,a_{n} ) \in\mathbb{Z}^{n}$ satisfying (\ref{eq.lem.B.wd1.eq}) belongs to~$S^{n}$. Since $S^{n}$ is a finite set, this shows that there exist only finitely many $( a_{1},a_{2},\ldots,a_{n}) \in\mathbb{Z}^{n}$ satisfying~(\ref{eq.lem.B.wd1.eq}). This proves Lemma~\ref{lem.B.wd1}.
\end{proof}

\begin{Definition}We define a continuous $\mathbf{k}$-linear map $B\colon \mathcal{Q}^{\pm} \rightarrow\mathcal{T}[[w]] $ by setting
\begin{gather*}
B\big( q_{1}^{a_{1}}q_{2}^{a_{2}}\cdots q_{n}^{a_{n}}\big) =\left(
\prod_{\substack{i\in[n] ;\\a_{i}>0}}t_{i}^{a_{i}}\right)
\left( \prod_{\substack{i\in[n] ;\\a_{i}<0}}w^{-a_{i}}\right)
\qquad \text{for each } ( a_{1},a_{2},\ldots ,a_{n} ) \in\mathbb{Z}^{n}.
\end{gather*}
This is well-defined, since $ ( q_{1}^{a_{1}}q_{2}^{a_{2}}\cdots q_{n}^{a_{n}} ) _{( a_{1},a_{2},\ldots,a_{n}) \in \mathbb{Z}^{n}}$ is a topological basis of $\mathcal{Q}^{\pm}$, and because of Lemma~\ref{lem.B.wd1} (which guarantees convergence when the map~$B$ is applied to an infinite $\mathbf{k}$-linear combination of Laurent monomials).

The $\mathbf{k}$-linear map $B\colon \mathcal{Q}^{\pm}\rightarrow\mathcal{T} [ [w] ] $ can be restricted to the $\mathbf{k}$-submodule $\mathcal{Q}$ of $\mathcal{Q}^{\pm}$. We denote this restriction by~$B$ as well. In the following, we shall only be concerned with this restriction.
\end{Definition}

Of course, $B$ is (in general) not a $\mathbf{k}$-algebra homomorphism.

\subsection[The $\mathbf{k}$-algebra monomorphism $\protect{E\colon \mathcal{T}^{\prime }\rightarrow\mathcal{T}^{\prime}[[w]]}$]{The $\boldsymbol{\mathbf{k}}$-algebra monomorphism $\boldsymbol{E\colon \mathcal{T}^{\prime }\rightarrow\mathcal{T}^{\prime}[[w]]}$}

\begin{Definition} We define a $\mathbf{k}$-algebra homomorphism $E\colon \mathcal{T}^{\prime}\rightarrow\mathcal{T}^{\prime}[[w]] $ by
\begin{gather*}
E(t_{i}) =-\frac{t_{i}+\beta+\alpha w}{1-t_{i}w} \qquad \text{for each }i\in[n-1] .
\end{gather*}
This is well-defined (by the universal property of the polynomial ring $\mathcal{T}^{\prime}=\mathbf{k}[ t_{1},t_{2},\ldots,t_{n-1}] $), because for each $i\in[n-1] $, the power series $1-t_{i}w$ is invertible in $\mathcal{T}^{\prime}[[w]] $ (indeed, its constant term is~$1$).
\end{Definition}

\begin{Proposition}\label{prop.E.inj}The homomorphism $E$ is injective.
\end{Proposition}

\begin{proof}Let $F\colon \mathcal{T}^{\prime} [ [w] ] \rightarrow\mathcal{T}^{\prime}$ be the $\mathcal{T}^{\prime}$-algebra homomorphism that sends each formal power
series $f\in\mathcal{T}^{\prime}[[w]] $ (regarded as a~formal power series in the single indeterminate~$w$ over $\mathcal{T} ^{\prime}$) to its constant term $f ( 0 ) \in\mathcal{T}^{\prime}$. Thus, $F$ is a~$\mathbf{k}$-algebra homomorphism, and it sends~$w$ to~$0$ while sending each element of $\mathcal{T}^{\prime}$ to itself.

Let $G\colon \mathcal{T}^{\prime}\rightarrow\mathcal{T}^{\prime}$ be the $\mathbf{k}$-algebra homomorphism that sends $t_{i}$ to $-t_{i}-\beta$ for each $i\in[n-1] $. (This is well-defined by the universal property of the polynomial ring $\mathcal{T}^{\prime}=\mathbf{k} [ t_{1},t_{2},\ldots,t_{n-1} ] $.)

Notice that the map $G\circ F\circ E\colon \mathcal{T}^{\prime}\rightarrow \mathcal{T}^{\prime}$ is a $\mathbf{k}$-algebra homomorphism (since it is the composition of the three $\mathbf{k}$-algebra homomorphisms~$E$,~$F$,~$G$).

For each $i\in[n-1] $, we have
\begin{gather*}
 ( F\circ E ) (t_{i}) =F\left( \underbrace{E(t_{i}) }_{=-\frac{t_{i}+\beta+\alpha w}{1-t_{i}
w}}\right) =F\left( -\frac{t_{i}+\beta+\alpha w}{1-t_{i}w}\right)
=-\frac{t_{i}+\beta+\alpha F(w) }{1-t_{i}F(w) }\\
\hphantom{( F\circ E ) (t_{i}) }{} \qquad ( \text{since }F\text{ is a }\mathcal{T}^{\prime}\text{-algebra homomorphism} ) \\
\hphantom{( F\circ E ) (t_{i}) }{} =-\frac{t_{i}+\beta+\alpha0}{1-t_{i}0}\qquad (\text{since }F(w) =0) \\
\hphantom{( F\circ E ) (t_{i}) }{} =- ( t_{i}+\beta )
\end{gather*}
and thus
\begin{gather*}
 ( G\circ F\circ E ) (t_{i}) =G\left(\underbrace{ ( F\circ E ) (t_{i}) }_{=- (t_{i}+\beta ) }\right) =G ( - ( t_{i}+\beta ) ) =-\left( \underbrace{G(t_{i}) }_{=-t_{i}-\beta}+\beta\right) \\
\hphantom{( G\circ F\circ E ) (t_{i})}{} \qquad ( \text{since }G\text{ is a }\mathbf{k}\text{-algebra homomorphism} ) \\
\hphantom{( G\circ F\circ E ) (t_{i})}{} =- (-t_{i}-\beta+\beta ) =t_{i}=\operatorname*{id} ( t_{i}) .
\end{gather*}
Hence, the two $\mathbf{k}$-algebra homomorphisms $G\circ F\circ E\colon \mathcal{T}^{\prime}\rightarrow\mathcal{T}^{\prime}$ and $\operatorname*{id}\colon \mathcal{T}^{\prime}\rightarrow\mathcal{T}^{\prime}$ agree on the generating set $\{ t_{1},t_{2},\ldots,t_{n-1}\} $ of the $\mathbf{k}$-algebra $\mathcal{T}^{\prime}$. Thus, these two homomorphisms must be identical. In other words, $G\circ F\circ E=\operatorname*{id}$. Hence, the map~$E$ has a left inverse, and thus is injective. This proves Proposition~\ref{prop.E.inj}.
\end{proof}

Thus, we have defined the following spaces and maps between them:
\begin{gather*}
\xymatrix{
\calX\ar[r]^{A} \ar[dr]_{D} & \calQ\ar[r]^-{B} & \calT\left[\left
[w\right]\right] \\
& \calT^\prime\ar[r]_-{E} & \calT^\prime\left[\left[w\right]\right] \arinj[u]
}
\end{gather*}
(but this is \textit{not} a commutative diagram). It is worth reminding ourselves that $A$, $D$ and $E$ are $\mathbf{k}$-algebra homomorphisms, but~$B$ (in general) is not.

\subsection[Pathless monomials and subsets $S$ of $\protect{[n-1]}$]{Pathless monomials and subsets $\boldsymbol{S}$ of $\boldsymbol{[n-1]}$}

Next, we want to study the action of the compositions $B\circ A$ and $E\circ D$ on pathless monomials. We first introduce some more notations:

\begin{Definition}\label{def.Ssubset}Let $S$ be a subset of $[n-1] $.
\begin{enumerate}\itemsep=0pt
\item[(a)] Let $\mathfrak{P}_{S}$ be the set of all pairs $(i,j) \in S\times( [n] \setminus S) $ satisfying $i<j$.

\item[(b)] A monomial $\mathfrak{m}\in\mathfrak{M}$ is said to be $S$\textit{-friendly} if it is a product of some of the indetermina\-tes~$x_{i,j}$ with $(i,j) \in\mathfrak{P}_{S}$. In other words, a~monomial $\mathfrak{m}\in\mathfrak{M}$ is $S$-friendly if and only if every indeterminate $x_{i,j}$ that appears in~$\mathfrak{m}$ satisfies $i\in S$ and~$j\notin S$.

We let $\mathfrak{M}_{S}$ denote the set of all $S$-friendly monomials.

\item[(c)] We let $\mathcal{X}_{S}$ denote the polynomial ring $\mathbf{k} [ x_{i,j}\,|\, (i,j) \in\mathfrak{P}_{S} ] $. This is clearly a subring of $\mathcal{X}$. The $\mathbf{k}$-module $\mathcal{X}_{S}$ has a basis consisting of all $S$-friendly monomials $\mathfrak{m}\in\mathfrak{M}$.

\item[(d)] An $n$-tuple $ ( a_{1},a_{2},\ldots,a_{n} ) \in\mathbb{Z}^{n}$ is said to be $S$\textit{-adequate} if and only if it satisfies \smash{$( a_{i}\geq0$} $\text{for all }i\in S ) $ and $ (a_{i}\leq0\text{ for all }i\in[n] \setminus S) $. We let $\mathcal{Q}_{S}$ denote the subset of $\mathcal{Q}$ consisting of all infinite $\mathbf{k}$-linear combinations of the Laurent monomials $q_{1}^{a_{1}}q_{2}^{a_{2}}\cdots q_{n}^{a_{n}}$ for $S$-adequate $n$-tuples $ ( a_{1},a_{2},\ldots,a_{n} ) \in\mathbb{Z}^{n}$ (as long as these combinations belong to $\mathcal{Q}$). It is easy to see that $\mathcal{Q}_{S}$ is a topological $\mathbf{k}$-subalgebra of $\mathcal{Q}$ (since the entrywise sum of two $S$-adequate $n$-tuples is $S$-adequate again).
\end{enumerate}
(At this point, it is helpful to recall once again that the $q_{1},q_{2},\ldots,q_{n}$ are not indeterminates, but rather monomials defined by $q_{i}=r_{i}r_{i+1}\cdots r_{n}$. But their products $q_{1}^{a_{1}}q_{2}^{a_{2}}\cdots q_{n}^{a_{n}}$ are Laurent monomials. Explicitly, they can be rewritten as products of the $r_{1},r_{2},\ldots,r_{n}$ using~(\ref{eq.Q'.reindex}). Thus, it is easy to see that the elements of~$\mathcal{Q}_{S}$ are the infinite $\mathbf{k}$-linear combinations of the Laurent monomials $r_{1}^{b_{1}}r_{2}^{b_{2}}\cdots r_{n}^{b_{n}}$ for all $( b_{1},b_{2},\ldots,b_{n}) \in\mathbb{Z}^{n}$ satisfying
$( b_{i}\geq b_{i-1}\text{ for all }i\in S) $ and $(b_{i}\leq b_{i-1}\text{ for all }i\in[n] \setminus S) $, where we set $b_{0}=0$, as long as these combinations belong to $\mathcal{Q}$.
But we won't need this characterization.)
\begin{enumerate}\itemsep=0pt
\item[(e)] We let $\mathcal{T}_{S}$ denote the topological $\mathbf{k}$-algebra $\mathbf{k} [ [ t_{i}\,|\, i\in S ] ] $. This is a topological subalgebra of~$\mathcal{T}$. Hence, the ring $\mathcal{T}_{S}[[w]] $ (that is, the ring of formal power series in the (single) variable~$w$ over $\mathcal{T}_{S}$) is a~topological $\mathbf{k}$-subalgebra of the similarly-defined ring~$\mathcal{T}[[w]] $.
\item[(f)] We define a $\mathbf{k}$-algebra homomorphism $A_{S}\colon \mathcal{X}_{S}\rightarrow\mathcal{Q}_{S}$ by
\begin{gather*}
A_{S} ( x_{i,j} ) =-\frac{q_{i}+\beta+\alpha/q_{j}}{1-q_{i}/q_{j}}\qquad \text{for all }(i,j) \in\mathfrak{P}_{S}.
\end{gather*}
This is well-defined, because for each $(i,j) \in\mathfrak{P}_{S}$, the power series $-\frac{q_{i}+\beta+\alpha/q_{j}}{1-q_{i}/q_{j}}$ does indeed belong to $\mathcal{Q}_{S}$ (indeed, a look at the monomials reveals that both series $-( q_{i}+\beta+\alpha/q_{j}) $ and $\frac{1}{1-q_{i}/q_{j}}=\sum_{k\geq0}( q_{i}/q_{j}) ^{k}$ belong to $\mathcal{Q}_{S}$, and therefore so does their product, which is $-\frac{q_{i}+\beta+\alpha/q_{j}}{1-q_{i}/q_{j}}$).

\item[(g)] We define a continuous $\mathbf{k}$-linear map $B_{S}\colon \mathcal{Q}_{S}\rightarrow\mathcal{T}_{S}[[w]] $ by setting
\begin{gather*}
B_{S}\big( q_{1}^{a_{1}}q_{2}^{a_{2}}\cdots q_{n}^{a_{n}}\big) =\left( \prod_{i\in S}t_{i}^{a_{i}}\right) \left( \prod_{i\in[n] \setminus S}w^{-a_{i}}\right) \\
\hphantom{B_{S}\big( q_{1}^{a_{1}}q_{2}^{a_{2}}\cdots q_{n}^{a_{n}}\big) =}{} \text{for each }S\text{-adequate } ( a_{1},a_{2},\ldots,a_{n}) \in\mathbb{Z}^{n}.
\end{gather*}
This is well-defined, as we will see below (in Proposition~\ref{prop.pathless.BSwd}(b)).
\item[(h)] We let $\mathcal{T}_{S}^{\prime}$ denote the $\mathbf{k}$-algebra $\mathbf{k} [ t_{i}\,|\, i\in S ] $. This is a $\mathbf{k}$-subalgebra of $\mathcal{T}^{\prime}$. Hence, the ring $\mathcal{T}_{S}^{\prime}[[w]] $ (that is, the ring of formal power series in the (single) variable~$w$ over $\mathcal{T}_{S}^{\prime}$) is a~$\mathbf{k}$-subalgebra of the similarly-defined ring~$\mathcal{T}^{\prime}[[w]] $.
\item[(i)] We define a $\mathbf{k}$-algebra homomorphism $D_{S}\colon \mathcal{X}_{S}\rightarrow\mathcal{T}_{S}^{\prime}$ by
\begin{gather*}
D_{S}(x_{i,j}) =t_{i}\qquad \text{for all} \quad (i,j) \in\mathfrak{P}_{S}.
\end{gather*}
This is well-defined, since each $(i,j) \in\mathfrak{P}_{S}$ satisfies $i\in S$.
\item[(j)] We define a $\mathbf{k}$-algebra homomorphism $E_{S}\colon \mathcal{T}_{S}^{\prime}\rightarrow\mathcal{T}_{S}^{\prime} [ [w]] $ by
\begin{gather*}
E_{S}(t_{i}) =-\frac{t_{i}+\beta+\alpha w}{1-t_{i}w}\qquad \text{for each }i\in S.
\end{gather*}
This is well-defined (by the universal property of the polynomial ring $\mathcal{T}_{S}^{\prime}$), because for each $i\in S$, the power series $1-t_{i}w$ is invertible in $\mathcal{T}_{S}^{\prime} [ [w]] $ (indeed, its constant term is~$1$).
\end{enumerate}
\end{Definition}

\begin{Proposition}\label{prop.pathless.BSwd}Let $S$ be a subset of $[n-1] $.
\begin{enumerate}\itemsep=0pt
\item[{\rm (a)}] We have
\begin{gather}
B\big( q_{1}^{a_{1}}q_{2}^{a_{2}}\cdots q_{n}^{a_{n}}\big) =\left(\prod_{i\in S}t_{i}^{a_{i}}\right) \left( \prod_{i\in[n]\setminus S}w^{-a_{i}}\right) \label{eq.prop.pathless.BSwd.a.eq}
\end{gather}
for each $S$-adequate $n$-tuple $ ( a_{1},a_{2},\ldots,a_{n} ) \in\mathbb{Z}^{n}$.

\item[{\rm (b)}] The map $B_{S}$ $($defined in Definition~{\rm \ref{def.Ssubset}(g))} is well-defined.
\end{enumerate}
\end{Proposition}

\begin{proof}(a) Let $(a_{1},a_{2},\ldots,a_{n}) \in\mathbb{Z}^{n}$ be an $S$-adequate $n$-tuple. We must prove~(\ref{eq.prop.pathless.BSwd.a.eq}).

The $n$-tuple $ ( a_{1},a_{2},\ldots,a_{n} ) $ is $S$-adequate. Thus, $ ( a_{i}\geq0\text{ for all }i\in S ) $ and $ (a_{i}\leq0\text{ for all }i\in[n] \setminus S ) $. Hence, each $i\in[n] $ satisfying $a_{i}>0$ must belong to $S$. Also, each $i\in[n] $ satisfying $a_{i}<0$ must belong to $ [n ] \setminus S$ (since $a_{i}\geq0$ for all $i\in S$).

Now, the definition of the map $B$ yields
\begin{gather*}
B\big( q_{1}^{a_{1}}q_{2}^{a_{2}}\cdots q_{n}^{a_{n}}\big)
=\underbrace{\left( \prod_{\substack{i\in[n] ;\\a_{i}>0}}t_{i}^{a_{i}}\right) }_{\substack{=\prod_{\substack{i\in S;\\a_{i}>0}
}t_{i}^{a_{i}}\\\text{(since each }i\in[n] \\\text{satisfying
}a_{i}>0\\\text{must belong to }S\text{)}}}\underbrace{\left( \prod
_{\substack{i\in[n] ;\\a_{i}<0}}w^{-a_{i}}\right)
}_{\substack{=\prod_{\substack{i\in[n] \setminus S;\\a_{i}
<0}}w^{-a_{i}}\\\text{(since each }i\in[n] \\\text{satisfying
}a_{i}<0\\\text{must belong to }[n] \setminus S\text{)}}}=\left( \prod_{\substack{i\in S;\\a_{i}>0}}t_{i}^{a_{i}}\right) \left(\prod_{\substack{i\in[n] \setminus S;\\a_{i}<0}}w^{-a_{i}}\right) .
\end{gather*}
Comparing this with
\begin{gather*}
 \underbrace{\left( \prod_{i\in S}t_{i}^{a_{i}}\right) }%
_{\substack{=\left( \prod_{\substack{i\in S;\\a_{i}=0}}t_{i}^{a_{i}}\right)
\left( \prod_{\substack{i\in S;\\a_{i}>0}}t_{i}^{a_{i}}\right)
\\\text{(since each }i\in S\text{ satisfies either }a_{i}=0\text{ or }%
a_{i}>0\\\text{(since }a_{i}\geq0\text{ for all }i\in S\text{))}%
}}\underbrace{\left( \prod_{i\in[n] \setminus S}w^{-a_{i}%
}\right) }_{\substack{=\left( \prod_{\substack{i\in[n]
\setminus S;\\a_{i}=0}}w^{-a_{i}}\right) \left( \prod_{\substack{i\in\left[
n\right] \setminus S;\\a_{i}<0}}w^{-a_{i}}\right) \\\text{(since each }%
i\in[n] \setminus S\text{ satisfies either }a_{i}=0\text{ or
}a_{i}<0\\\text{(since }a_{i}\leq0\text{ for all }i\in[n]
\setminus S\text{))}}}\\
\qquad{} =\left( \prod_{\substack{i\in S;\\a_{i}=0}}\underbrace{t_{i}^{a_{i}}%
}_{\substack{=1\\\text{(since }a_{i}=0\text{)}}}\right) \left(
\prod_{\substack{i\in S;\\a_{i}>0}}t_{i}^{a_{i}}\right) \left(
\prod_{\substack{i\in[n] \setminus S;\\a_{i}=0}%
}\underbrace{w^{-a_{i}}}_{\substack{=1\\\text{(since }a_{i}=0\text{)}%
}}\right) \left( \prod_{\substack{i\in[n] \setminus
S;\\a_{i}<0}}w^{-a_{i}}\right) \\
\qquad{} =\left( \prod_{\substack{i\in S;\\a_{i}>0}}t_{i}^{a_{i}}\right) \left(
\prod_{\substack{i\in[n] \setminus S;\\a_{i}<0}}w^{-a_{i}}\right) ,
\end{gather*}
we obtain $B\big( q_{1}^{a_{1}}q_{2}^{a_{2}}\cdots q_{n}^{a_{n}}\big) =\left( \prod_{i\in S}t_{i}^{a_{i}}\right) \left( \prod_{i\in[n] \setminus S}w^{-a_{i}}\right) $. This proves Proposition~\ref{prop.pathless.BSwd}(a).

(b) We must show that there exists a unique continuous $\mathbf{k}$-linear map $B_{S}\colon \mathcal{Q}_{S}\rightarrow\mathcal{T}_{S}[[w]] $ satisfying%
\begin{gather}
B_{S}\big( q_{1}^{a_{1}}q_{2}^{a_{2}}\cdots q_{n}^{a_{n}}\big) =\left(\prod_{i\in S}t_{i}^{a_{i}}\right) \left( \prod_{i\in[n]\setminus S}w^{-a_{i}}\right) \label{pf.prop.pathless.BSwd.b.cond}\\
\hphantom{B_{S}\big( q_{1}^{a_{1}}q_{2}^{a_{2}}\cdots q_{n}^{a_{n}}\big) =}{} \text{for each }S\text{-adequate } ( a_{1},a_{2},\ldots,a_{n} ) \in\mathbb{Z}^{n}.\nonumber
\end{gather}
The uniqueness of such a map is clear (because the elements of $\mathcal{Q}_{S}$ are infinite $\mathbf{k}$-linear combinations of the Laurent monomials $q_{1}^{a_{1}}q_{2}^{a_{2}}\cdots q_{n}^{a_{n}}$ for $S$-adequate $n$-tuples $( a_{1},a_{2},\ldots,a_{n}) \in\mathbb{Z}^{n}$; but the formula~(\ref{pf.prop.pathless.BSwd.b.cond}) uniquely determines the value of $B_{S}$ on such a $\mathbf{k}$-linear combination). Thus, it remains to prove its existence.

For each $f\in\mathcal{Q}_{S}$, we have $B(f) \in\mathcal{T}_{S}[[w]] $.\footnote{\textbf{Proof.} Let $f\in\mathcal{Q}_{S}$. We must show that $B(f) \in \mathcal{T}_{S}[[w]] $. Since the map $B$ is $\mathbf{k}$-linear and continuous, we can WLOG assume that $f$ is a Laurent monomial of the form $q_{1}^{a_{1}}q_{2}^{a_{2}}\cdots q_{n}^{a_{n}}$ for some $S$-adequate $n$-tuple $\left( a_{1},a_{2},\ldots,a_{n}\right) \in \mathbb{Z}^{n}$ (because $f$ is always an infinite $\mathbf{k}$-linear combination of such Laurent monomials). Assume this. Consider this $ (a_{1},a_{2},\ldots,a_{n}) \in\mathbb{Z}^{n}$.

Thus, $f=q_{1}^{a_{1}}q_{2}^{a_{2}}\cdots q_{n}^{a_{n}}$. Applying the map $B$ to both sides of this equality, we obtain
\begin{gather*}
B(f) =B\big( q_{1}^{a_{1}}q_{2}^{a_{2}}\cdots q_{n}^{a_{n}}\big) =\left( \prod_{i\in S}t_{i}^{a_{i}}\right) \left(\prod_{i\in[n] \setminus S}w^{-a_{i}}\right) \qquad ( \text{by Proposition \ref{prop.pathless.BSwd}(a)}) \\
\hphantom{B(f) =}{} \in\mathcal{T}_{S}[[w]] .
\end{gather*}
This is precisely what we wanted to show.} Hence, we can define a map $\widetilde{B_{S}}\colon \mathcal{Q}_{S}\rightarrow\mathcal{T}_{S} [ [w]] $ by
\begin{gather*}
\widetilde{B_{S}}(f) =B(f) \qquad \text{for each }f\in\mathcal{Q}_{S}.
\end{gather*}
This map $\widetilde{B_{S}}$ is a restriction of the map $B$; hence, it is a~continuous $\mathbf{k}$-linear map (since $B$ is a~continuous $\mathbf{k}$-linear map). Furthermore, it satisfies
\begin{gather*}
\widetilde{B_{S}}\big( q_{1}^{a_{1}}q_{2}^{a_{2}}\cdots q_{n}^{a_{n}}\big) =B\big( q_{1}^{a_{1}}q_{2}^{a_{2}}\cdots q_{n}^{a_{n}}\big)\qquad ( \text{by the definition of }\widetilde{B_{S}}) \\
\hphantom{\widetilde{B_{S}}\big( q_{1}^{a_{1}}q_{2}^{a_{2}}\cdots q_{n}^{a_{n}}\big)}{} =\left( \prod_{i\in S}t_{i}^{a_{i}}\right) \left( \prod_{i\in [n] \setminus S}w^{-a_{i}}\right) \qquad ( \text{by Proposition \ref{prop.pathless.BSwd}(a)})
\end{gather*}
for each $S$-adequate $ ( a_{1},a_{2},\ldots,a_{n} ) \in \mathbb{Z}^{n}$. Hence, $\widetilde{B_{S}}$ is a continuous $\mathbf{k}$-linear map $B_{S}\colon \mathcal{Q}_{S}\rightarrow\mathcal{T}_{S} [ [w]] $ satisfying~(\ref{pf.prop.pathless.BSwd.b.cond}). Thus, the existence of such a map $B_{S}$ is proven. As we have explained, this completes the proof of Proposition~\ref{prop.pathless.BSwd}(b).
\end{proof}

\begin{Proposition}\label{prop.pathless.cd}Let $S$ be a subset of $[n-1] $. Then,
the diagrams%
\begin{gather}\begin{split} &
\xymatrix{
\calX_S \arinj[d] \ar[r]^{A_S} & \calQ_S \arinj[d] \ar[r]^-{B_S}
& \calT_S\left[\left[w\right]\right] \arinj[d] \\
\calX\ar[r]_{A} & \calQ\ar[r]_-{B} & \calT\left[\left[w\right]\right]
}\end{split}
\label{eq.prop.pathless.cd.1}
\end{gather}
and%
\begin{gather}\begin{split}&
\xymatrix{
\calX_S \arinj[d] \ar[r]^{D_S} & \calT^\prime_S \arinj[d] \ar[r]^-{E_S}
& \calT^\prime_S\left[\left[w\right]\right] \arinj[d] \\
\calX\ar[r]_{D} & \calT^\prime\ar[r]_-{E} & \calT^\prime\left[\left
[w\right]\right]
}\end{split}
\label{eq.prop.pathless.cd.2}
\end{gather}
$($where the vertical arrows are the obvious inclusion maps$)$ are commutative.
\end{Proposition}

\begin{proof}The commutativity of the left square of~(\ref{eq.prop.pathless.cd.1}) is obvious\footnote{``Obvious'' in the following sense: You want to prove that a~diagram of the form
\begin{gather*}
\xymatrix{
\mathcal{A}_1 \ar[r]^{f_1} \ar[d]_{f_2} & \mathcal{A}_2 \ar[d]^{f_3} \\
\mathcal{A}_3 \ar[r]_{f_4} & \mathcal{A}_4
}
\end{gather*}
is commutative, where $\mathcal{A}_{1}$, $\mathcal{A}_{2}$, $\mathcal{A}_{3}$, $\mathcal{A}_{4}$ are four $\mathbf{k}$-algebras and $f_{1}$, $f_{2}$, $f_{3}$, $f_{4}$ are four $\mathbf{k}$-algebra homomorphisms. (In our concrete case, $\mathcal{A}_{1}=\mathcal{X}_{S}$, $\mathcal{A}_{2}=\mathcal{Q}_{S}$, $\mathcal{A}_{3}=\mathcal{X}$, $\mathcal{A}_{4}=\mathcal{Q}$, $f_{1}=A_{S}$ and $f_{4}=A$, whereas $f_{2}$ and $f_{3}$ are the inclusion maps $\mathcal{X}_{S}\rightarrow\mathcal{X}$ and $\mathcal{Q}_{S}\rightarrow \mathcal{Q}$.) In order to prove this commutativity, it suffices to show that it holds \textit{on a generating set} of the $\mathbf{k}$-algebra
$\mathcal{A}_{1}$. In other words, it suffices to pick some generating set $\mathfrak{G}$ of the $\mathbf{k}$-algebra~$\mathcal{A}_{1}$ and show that all $g\in\mathfrak{G}$ satisfy $( f_{3}\circ f_{1}) ( g)= ( f_{4}\circ f_{2} ) ( g ) $. (In our concrete case, it is most reasonable to pick $\mathfrak{G}= \{ x_{i,j}\,|\, (i,j ) \in\mathfrak{P}_{S} \} $. The proof then becomes completely clear.)}. So is the commutativity of each of the two squares of~(\ref{eq.prop.pathless.cd.2})\footnote{For similar reasons.}. It thus remains to prove the commutativity of the right square of~(\ref{eq.prop.pathless.cd.1}). In other words, we must show that $B_{S}(p) =B(p) $ for each $p\in\mathcal{Q}_{S}$.

So fix $p\in\mathcal{Q}_{S}$. Since both maps $B_{S}$ and $B$ are continuous and $\mathbf{k}$-linear, we can WLOG assume that $p$ is a Laurent monomial of the form $q_{1}^{a_{1}}q_{2}^{a_{2}}\cdots q_{n}^{a_{n}}$ for an $S$-adequate $n$-tuple $( a_{1},a_{2},\ldots,a_{n}) \in\mathbb{Z}^{n}$ (since the elements of $\mathcal{Q}_{S}$ are infinite $\mathbf{k}$-linear combinations of Laurent monomials of this form). Assume this, and fix this
$ ( a_{1},a_{2},\ldots,a_{n} ) $.

From $p=q_{1}^{a_{1}}q_{2}^{a_{2}}\cdots q_{n}^{a_{n}}$, we obtain
\begin{gather*}
B(p) =B\big( q_{1}^{a_{1}}q_{2}^{a_{2}}\cdots q_{n}^{a_{n}}\big) =\left( \prod_{i\in S}t_{i}^{a_{i}}\right) \left( \prod _{i\in[n] \setminus S}w^{-a_{i}}\right)
\end{gather*}
(by Proposition \ref{prop.pathless.BSwd}(a)). Comparing this with
\begin{gather*}
B_{S}(p)=B_{S}\big( q_{1}^{a_{1}}q_{2}^{a_{2}}\cdots q_{n}^{a_{n}}\big) \qquad ( \text{since }p=q_{1}^{a_{1}}q_{2}^{a_{2}}\cdots q_{n}^{a_{n}}) \\
\hphantom{B_{S}(p)}{} =\left( \prod_{i\in S}t_{i}^{a_{i}}\right) \left( \prod_{i\in [n] \setminus S}w^{-a_{i}}\right) \qquad ( \text{by the definition of }B_{S}) ,
\end{gather*}
we obtain $B_{S}(p) =B(p) $. This proves the commutativity of the right square of~(\ref{eq.prop.pathless.cd.1}). The proof of Proposition~\ref{prop.pathless.cd} is thus complete.
\end{proof}

\begin{Proposition}\label{prop.pathless.BS-alg}Let $S$ be a subset of $[n-1] $. Then, $B_{S}\colon \mathcal{Q}_{S}\rightarrow\mathcal{T}_{S} [ [ w] ] $ is a continuous $\mathbf{k}$-algebra homomorphism.
\end{Proposition}

\begin{proof}We merely need to show that $B_{S}$ is a $\mathbf{k}$-algebra homomorphism. To this purpose, by linearity, we only need to prove that $B_{S}(1) =1$ and $B_{S}( \mathfrak{mn}) =B_{S}( \mathfrak{m}) B_{S}(\mathfrak{n}) $ for any two Laurent monomials $\mathfrak{m}$ and $\mathfrak{n}$ of the form $q_{1}^{a_{1}}q_{2}^{a_{2}}\cdots q_{n}^{a_{n}}$ for $S$-adequate $n$-tuples $( a_{1},a_{2},\ldots,a_{n}) \in\mathbb{Z}^{n}$ (since the elements of $\mathcal{Q}_{S}$ are infinite $\mathbf{k}$-linear combinations of Laurent monomials of this form). This is easy and left to the reader.
\end{proof}

\begin{Proposition}\label{prop.pathless.AB}Let $S$ be a subset of $[n-1] $. Let $(i,j) \in\mathfrak{P}_{S}$. Then,
\begin{gather*}
 ( E_{S}\circ D_{S} ) (x_{i,j}) = ( B_{S}\circ A_{S} ) ( x_{i,j} ) .
\end{gather*}
\end{Proposition}

\begin{proof}From $(i,j) \in\mathfrak{P}_{S}$, we obtain $i\in S$ and $j\in[n] \setminus S$. Thus, the definition of $B_{S}$ reveals that $B_{S} ( q_{i} ) =t_{i}$ and $B_{S}\big( q_{j}^{-1}\big) =w$.

Proposition \ref{prop.pathless.BS-alg} shows that $B_{S}\colon \mathcal{Q}_{S}\rightarrow\mathcal{T}_{S}[[w]] $ is a~continuous $\mathbf{k}$-algebra homomorphism. But the definition of~$A_{S}$ yields
\begin{gather*}
A_{S}(x_{i,j}) =-\frac{q_{i}+\beta+\alpha/q_{j}}{1-q_{i}/q_{j}}=-\frac{q_{i}+\beta+\alpha q_{j}^{-1}}{1-q_{i}q_{j}^{-1}}.
\end{gather*}
Applying the map $B_{S}$ to both sides of this equality, we find
\begin{gather*}
B_{S} ( A_{S}(x_{i,j}) ) =B_{S}\left( -\frac{q_{i}+\beta+\alpha q_{j}^{-1}}{1-q_{i}q_{j}^{-1}}\right)=-\frac{B_{S} ( q_{i} ) +\beta+\alpha B_{S} \big( q_{j}
^{-1}\big) }{1-B_{S}(q_{i}) B_{S}\big( q_{j}^{-1}\big)}\\
\hphantom{B_{S} ( A_{S}(x_{i,j}) )}{} \qquad \left(
\begin{matrix}
\text{since }B_{S}\text{ is a }\mathbf{k}\text{-algebra homomorphism, and
thus}\\
\text{respects sums, products and fractions (as long as}\\
\text{the denominators of the fractions are invertible)}%
\end{matrix}
\right) \\
\hphantom{B_{S} ( A_{S}(x_{i,j}) )}{} =-\frac{t_{i}+\beta+\alpha w}{1-t_{i}w}\qquad \big(\text{since }B_{S}(q_{i}) =t_{i}\text{ and }B_{S}\big(q_{j}^{-1}\big) =w\big) .
\end{gather*}
Comparing this with
\begin{gather*}
( E_{S}\circ D_{S}) (x_{i,j}) =E_{S}\Big(\underbrace{D_{S}(x_{i,j}) }_{=t_{i}}\Big) =E_{S} (t_{i}) =-\frac{t_{i}+\beta+\alpha w}{1-t_{i}w},
\end{gather*}
we obtain $B_{S} ( A_{S}(x_{i,j}) ) = ( E_{S}\circ D_{S}) (x_{i,j}) $. Thus,
\begin{gather*}
( E_{S}\circ D_{S} ) (x_{i,j}) =B_{S} ( A_{S} ( x_{i,j} ) ) = ( B_{S}\circ A_{S} ) (x_{i,j} ) .
\end{gather*}
This proves Proposition \ref{prop.pathless.AB}.
\end{proof}

\begin{Proposition}\label{prop.pathless.S}Let $\mathfrak{m}\in\mathfrak{M}$ be a pathless monomial.
\begin{enumerate}\itemsep=0pt
\item[{\rm (a)}] There exists a subset $S$ of $[n-1] $ such that $\mathfrak{m}$ is $S$-friendly.
\item[{\rm (b)}] Let $S$ be such a subset. Then, $\mathfrak{m}\in\mathcal{X}_{S}$ and $(E\circ D) (\mathfrak{m}) = ( B_{S}\circ A_{S}) (\mathfrak{m}) $.
\end{enumerate}
\end{Proposition}

\begin{proof} (a) Write $\mathfrak{m}$ in the form $\mathfrak{m}=\prod_{\substack{(i,j) \in [n] ^{2};\\i<j}}x_{i,j}^{a_{i,j}}$. For each $i\in[n-1] $, define a $b_{i}\in\mathbb{N}$ by $b_{i}=\sum_{j=i+1}^{n}a_{i,j}$. Define a subset $S$ of $[n-1] $ by $S= \{ i\in[n-1] \,|\, b_{i}>0 \} $. Then $\mathfrak{m}$ is $S$-friendly\footnote{{\bf Proof.} We need to show that every indeterminate $x_{i,j}$ that appears in $\mathfrak{m}$ satisfies $i\in S$ and $j\notin S$.

Indeed, assume the contrary. Thus, some indeterminate $x_{i,j}$ that appears in $\mathfrak{m}$ does \textit{not} satisfy $i\in S$ and $j\notin S$. Fix such an indeterminate $x_{i,j}$, and denote it by $x_{u,v}$. Thus, $x_{u,v}$ is an indeterminate that appears in $\mathfrak{m}$ but does \textit{not} satisfy $u\in S$ and $v\notin S$. Therefore, we have either $u\notin S$ or $v\in S$ (or both).

We have $1\leq u<v\leq n$ (since the indeterminate $x_{u,v}$ exists) and thus $u\in[n-1] $. The definition of $b_{u}$ yields $b_{u}=\sum_{j=u+1}^{n}a_{u,j}$. But $v\geq u+1$ (since $u<v$). Hence, $a_{u,v}$ is an addend of the sum $\sum_{j=u+1}^{n}a_{u,j}$. Hence, $\sum_{j=u+1}^{n}a_{u,j}\geq a_{u,v}$. But $a_{u,v}>0$ (since the indeterminate $x_{u,v}$ appears in $\mathfrak{m}$). Hence, $b_{u}=\sum_{j=u+1}^{n}a_{u,j}\geq a_{u,v}>0$. Therefore, $u\in S$ (by the definition of $S$). Hence, $u\notin S$ cannot hold. Therefore, $v\in S$ (since we know that we have either $u\notin S$ or $v\in S$). In other words, $v\in[n-1] $ and $b_{v}>0$ (by the definition of $S$). But the definition of $b_{v}$ yields $b_{v}=\sum_{j=v+1}^{n}a_{v,j}=\sum_{w=v+1}^{n}a_{v,w}$. Hence, $\sum_{w=v+1}^{n}a_{v,w}=b_{v}>0$. Hence, there exists some $w\in \{ v+1,v+2,\ldots,n \} $ such that $a_{v,w}>0$. Fix such a~$w$.

We have $v<w$ (since $w\in \{ v+1,v+2,\ldots,n \} $), hence $u<v<w$. Thus, $(u,v) \neq(v,w) $. Moreover, the indeterminate $x_{v,w}$ appears in $\mathfrak{m}$ (since $a_{v,w}>0$). Thus, both indeterminates $x_{u,v}$ and $x_{v,w}$ appear in $\mathfrak{m}$. Hence, $x_{u,v}x_{v,w}\,|\, \mathfrak{m}$ (since $(u,v) \neq ( v,w ) $).

But the monomial $\mathfrak{m}$ is pathless. In other words, there exists no triple $(i,j,k) \in[n] ^{3}$ satisfying $i<j<k$ and $x_{i,j}x_{j,k}\,|\, \mathfrak{m}$. This contradicts the fact that $(u,v,w) $ is such a triple (since $u<v<w$ and $x_{u,v}x_{v,w}\,|\, \mathfrak{m}$). This contradiction completes our proof.}. This proves Proposition \ref{prop.pathless.S}(a).

(b) We know that $\mathfrak{m}\in\mathcal{X}_{S}$ (since $\mathfrak{m}$ is $S$-friendly). Now, we shall show that $ ( E\circ D ) \,|\, _{\mathcal{X}_{S}}=B_{S}\circ A_{S}$ (if we regard $B_{S}\circ A_{S}$ as a map to $\mathcal{T}[[w]] $ and regard $E\circ D$ as a map to $\mathcal{T}[[w]] $).

The map $(E\circ D)|_{\mathcal{X}_{S}}$ is a $\mathbf{k}$-algebra homomorphism (since $D$ and $E$ are $\mathbf{k}$-algebra homomorphisms), and the map $B_{S}\circ A_{S}$ is a $\mathbf{k}$-algebra homomorphism (since both $B_{S}$ and $A_{S}$ are $\mathbf{k}$-algebra homomorphisms\footnote{Here we are using Proposition~\ref{prop.pathless.BS-alg}.}). Hence, we are trying to prove that two $\mathbf{k}$-algebra homomorphisms are equal (namely, the homomorphisms $(E\circ D)|_{\mathcal{X}_{S}}$ and $B_{S}\circ A_{S}$). It is clearly enough to prove this on the generating family $ (x_{i,j}) _{(i,j) \in\mathfrak{P}_{S}}$ of the $\mathbf{k}$-algebra $\mathcal{X}_{S}$. In other words, it is enough to prove that $ ( (E\circ D)|_{\mathcal{X}_{S}} ) ( x_{i,j}) = ( B_{S}\circ A_{S} ) (x_{i,j}) $ for each $(i,j) \in\mathfrak{P}_{S}$.

So let us fix some $(i,j) \in\mathfrak{P}_{S}$. Proposition~\ref{prop.pathless.cd} shows that the diagram~(\ref{eq.prop.pathless.cd.2}) is commutative. Thus, $(E\circ D)|_{\mathcal{X}_{S}}=E_{S}\circ D_{S}$ (provided that we regard $E_{S}\circ D_{S}$ as a map to $\mathcal{T}^{\prime}[[w]] $), and thus
\begin{gather*}
 ( (E\circ D)|_{\mathcal{X}_{S}} ) ( x_{i,j} ) = ( E_{S}\circ D_{S} ) (x_{i,j}) =( B_{S}\circ A_{S}) (x_{i,j})
\end{gather*}
(by Proposition \ref{prop.pathless.AB}).

This completes our proof of $(E\circ D) |_{\mathcal{X}_{S}}=B_{S}\circ A_{S}$.

Now, from $\mathfrak{m}\in\mathcal{X}_{S}$, we obtain $(E\circ D) (\mathfrak{m}) =\underbrace{ ((E\circ D) |_{\mathcal{X}_{S}}) }_{=B_{S}\circ A_{S}}(\mathfrak{m}) = ( B_{S}\circ A_{S} ) (\mathfrak{m}) $. This completes the proof of Proposition~\ref{prop.pathless.S}(b).
\end{proof}

\subsection[$(E\circ D) (q) =( B\circ A) (q) $ for pathless $q$]{$\boldsymbol{(E\circ D) (q) =( B\circ A) (q)}$ for pathless $\boldsymbol{q}$}

\begin{Corollary}\label{cor.pathless.D}Let $q\in\mathcal{X}$ be pathless. Then, $( E\circ D ) (q) = ( B\circ A ) (q) $.
\end{Corollary}

\begin{proof}The polynomial $q$ is pathless, i.e., is a $\mathbf{k}$-linear combination of pathless monomials. Hence, we WLOG assume that $q$ is a pathless monomial $\mathfrak{m}$ (since both maps $E\circ D$ and $B\circ A$ are $\mathbf{k}$-linear). Consider this~$\mathfrak{m}$.

Proposition \ref{prop.pathless.S}(a) shows that there exists a subset $S$ of $[n-1] $ such that $\mathfrak{m}$ is $S$-friendly. Consider this~$S$.

Proposition \ref{prop.pathless.S}(b) yields $\mathfrak{m}\in\mathcal{X}_{S}$ and $(E\circ D) (\mathfrak{m}) = ( B_{S}\circ A_{S} ) (\mathfrak{m}) $. But the commutativity of the diagram~(\ref{eq.prop.pathless.cd.1}) in Proposition~\ref{prop.pathless.cd} shows that $B_{S}\circ A_{S}= ( B\circ A ) |_{\mathcal{X}_{S}}$ (provided that we regard $B_{S}\circ A_{S}$ as a map to~$\mathcal{T}[[w]] $). Hence,
\begin{gather*}
\underbrace{( B_{S}\circ A_{S}) }_{=( B\circ A)|_{\mathcal{X}_{S}}}(\mathfrak{m}) = (( B\circ A ) |_{\mathcal{X}_{S}} ) (\mathfrak{m}) = (B\circ A ) (\mathfrak{m}) .
\end{gather*}
Thus, $(E\circ D) (\mathfrak{m}) = ( B_{S}\circ A_{S}) (\mathfrak{m}) =( B\circ A ) (\mathfrak{m}) $. Since $q=\mathfrak{m}$, this rewrites as $(E\circ D) (q) = ( B\circ A ) (q) $. This proves Corollary~\ref{cor.pathless.D}.
\end{proof}

\subsection{Proof of Theorem~\ref{thm.t-red.unique}}

\begin{Lemma}\label{lem.t-red.unique.0}Let $p\in\mathcal{X}$ be a pathless polynomial such that $p\in\mathcal{J}$. Then, $D(p) =0$.
\end{Lemma}

\begin{proof}We have $A \big( \underbrace{p}_{\in\mathcal{J}} \big) \in A ( \mathcal{J} ) =0$ (by Proposition~\ref{prop.A.killsJ}); thus, $A(p) =0$. But Corollary~\ref{cor.pathless.D} (applied to $q=p$) yields
\begin{gather*}
(E\circ D) (p) =(B\circ A) (p) =B \Big( \underbrace{A(p) }_{=0} \Big) =B ( 0 ) =0
\end{gather*}
(since the map $B$ is $\mathbf{k}$-linear). Thus, $E ( D(p)) =(E\circ D) (p) =0$. Since the $\mathbf{k}$-linear map $E$ is injective (by Proposition~\ref{prop.E.inj}), we thus conclude that $D(p) =0$. This proves Lemma~\ref{lem.t-red.unique.0}.
\end{proof}

We are now ready to prove Theorem~\ref{thm.t-red.unique}:

\begin{proof}[Proof of Theorem~\ref{thm.t-red.unique}]We need to prove that $D ( q) $ does not depend on the choice of $q$. In other words, we need to prove that if $f$ and $g$ are two pathless polynomials $q\in\mathcal{X}$ such that $p\equiv q\operatorname{mod}\mathcal{J}$, then $D(f) =D(g) $.

So let $f$ and $g$ be two pathless polynomials $q\in\mathcal{X}$ such that $p\equiv q\operatorname{mod}\mathcal{J}$. Thus, $p\equiv f\operatorname{mod} \mathcal{J}$ and $p\equiv g\operatorname{mod}\mathcal{J}$. Hence, $f\equiv p\equiv g\operatorname{mod}\mathcal{J}$, so that $f-g\in\mathcal{J}$. Also, the polynomial $f-g\in\mathcal{X}$ is pathless (since it is the difference of the two pathless polynomials~$f$ and~$g$). Thus, Lemma~\ref{lem.t-red.unique.0} (applied to $f-g$ instead of $p$) shows that $D(f-g) =0$. Thus, $0=D(f-g) =D(f) -D(g) $ (since $D$ is a $\mathbf{k}$-algebra homomorphism). In other words, $D(f) =D(g) $. This proves
Theorem~\ref{thm.t-red.unique}.
\end{proof}

\subsection[Appendix: A symmetric description of $\mathcal{J}$]{Appendix: A symmetric description of $\boldsymbol{\mathcal{J}}$}\label{sect.symmetry}

In this section, let us give a different description of $\mathcal{J}$ that reveals a symmetry inherent in the setting. First, we introduce auxiliary polynomials. So far, we have only been considering indeterminates $x_{i,j}$
corresponding to pairs $(i,j) \in[n] ^{2}$ satisfying $i<j$. We shall now also define~$x_{i,j}$ for pairs $ (i,j) \in[n] ^{2}$ satisfying $i>j$; but these~$x_{i,j}$ will not be new indeterminates, but rather will be polynomials in~$\mathcal{X}$:

\begin{Definition}\label{def.symmetry.xji}\quad\begin{enumerate}\itemsep=0pt
\item[(a)] Let $(i,j) \in [ n] ^{2}$ be a pair satisfying $i>j$. Then, we define an element $x_{i,j}\in\mathcal{X}$ by $x_{i,j}=-\beta-x_{j,i}$.
Thus, an element $x_{i,j}\in\mathcal{X}$ is defined for any pair $ (i,j) $ of two distinct elements of~$[n]$.
\item[(b)] For any three distinct elements $i$, $j$, $k$ of $[n] $, we define a polynomial $J_{i,j,k}\in\mathcal{X}$ by
\begin{gather*}
J_{i,j,k}=x_{i,j}x_{j,k}+x_{j,k}x_{k,i}+x_{k,i}x_{i,j}+\beta ( x_{i,j}+x_{j,k}+x_{k,i}) +\beta^{2}-\alpha.
\end{gather*}
\end{enumerate}
\end{Definition}

\begin{Proposition}\label{prop.symmetry}The ideal $\mathcal{J}$ of $\mathcal{X}$ is generated by all polynomials $J_{i,j,k}$ for $i$, $j$, $k$ being three distinct elements of~$[n] $.
\end{Proposition}

\begin{proof}If $(u_{p}) _{p\in P}$ is any family of elements of $\mathcal{X}$, then $ \langle u_{p} \rangle _{p\in P}$ shall mean the ideal of $\mathcal{X}$ generated by this family $ ( u_{p} ) _{p\in P}$. Thus, we need to prove that
\begin{gather*} \mathcal{J}= \langle J_{i,j,k} \rangle _{i,j,k\text{ are three distinct elements of }[n] }.\end{gather*}

We know (from the definition of $\mathcal{J}$) that
\begin{gather}
\mathcal{J}= \langle x_{i,j}x_{j,k}-x_{i,k} ( x_{i,j}+x_{j,k} +\beta ) -\alpha \rangle _{(i,j,k) \in [n ] ^{3}\text{ satisfying }i<j<k}. \label{pf.prop.symmetry.J=1}
\end{gather}

But for each $(i,j,k) \in[n] ^{3}$ satisfying $i<j<k$, a straightforward computation reveals that
\begin{gather*}
x_{i,j}x_{j,k}-x_{i,k} ( x_{i,j}+x_{j,k}+\beta ) -\alpha =J_{i,j,k}.
\end{gather*}
Hence, (\ref{pf.prop.symmetry.J=1}) rewrites as follows:
\begin{gather}
\mathcal{J}= \langle J_{i,j,k} \rangle _{(i,j,k) \in[n] ^{3}\text{ satisfying }i<j<k}. \label{pf.prop.symmetry.J=3}
\end{gather}

On the other hand, the definition of $J_{i,j,k}$ shows that $J_{i,j,k}$ is symmetric in its three arguments $i$, $j$, $k$; in other words, we have
\begin{gather*}
J_{i,j,k}=J_{i,k,j}=J_{j,i,k}=J_{j,k,i}=J_{k,i,j}=J_{k,j,i}
\end{gather*}
whenever $i$, $j$, $k$ are three distinct elements of $[n] $. Thus,
\begin{gather*}
 \langle J_{i,j,k} \rangle _{i,j,k\text{ are three distinct elements of }[n] }= \langle J_{i,j,k} \rangle _{(i,j,k) \in[n] ^{3}\text{ satisfying }i<j<k}.
\end{gather*}
Comparing this with (\ref{pf.prop.symmetry.J=3}), we obtain $\mathcal{J}=\left\langle J_{i,j,k}\right\rangle _{i,j,k\text{ are three distinct elements of }[n] }$. This proves Proposition~\ref{prop.symmetry}.
\end{proof}

Proposition \ref{prop.symmetry} reveals a hidden symmetry in the definitions of $\mathcal{X}$ and~$\mathcal{J}$:

\begin{Proposition}\label{prop.symmetry2}Consider the symmetric group $S_{n}$ $($that is, the group of all permutations of~$[n] )$.
\begin{enumerate}\itemsep=0pt
\item[{\rm (a)}] There is a unique action of the group $S_{n}$ on $\mathcal{X}$ by $\mathbf{k}$-algebra automorphisms satisfying
\begin{gather*}
\sigma\cdot x_{i,j}=x_{\sigma(i) ,\sigma(j)}\qquad \text{for all }\sigma\in S_{n}\ \text{and all pairs }(i,j) \text{ of distinct elements of }[n] .
\end{gather*}
\item[{\rm (b)}] The ideal $\mathcal{J}$ is invariant under this action of $S_{n}$, and thus the quotient $\mathbf{k}$-algebra $\mathcal{X}/\mathcal{J}$ inherits this action of~$S_{n}$.
\end{enumerate}
\end{Proposition}

\begin{proof}(a) Let $\mathcal{Y}$ be the polynomial ring
\begin{gather*}
\mathbf{k}\big[ y_{i,j}\,|\, (i,j) \in[n]^{2}\text{ such that }i\neq j\big] .
\end{gather*}
This is a polynomial ring in $n(n-1) $ indeterminates $y_{i,j}$ over $\mathbf{k}$. The symmetric group $S_{n}$ acts on $\mathcal{Y}$ by $\mathbf{k}$-algebra automorphisms; this action is defined by
\begin{gather*}
\sigma\cdot y_{i,j}=y_{\sigma(i) ,\sigma(j)}\qquad \text{for all }\sigma\in S_{n}\ \text{and all pairs }(i,j) \text{ of distinct elements of }[n] .
\end{gather*}
(The well-definedness of this action follows easily from the universal property of the polynomial ring $\mathcal{Y}$.)

Let $\phi\colon \mathcal{Y}\rightarrow\mathcal{X}$ be the unique $\mathbf{k} $-algebra homomorphism that sends each $y_{i,j}$ to $x_{i,j}$. This $\phi$ is well-defined by the universal property of the polynomial ring~$\mathcal{Y}$. Also, $\phi$ is surjective, since the generators $x_{i,j}$ of $\mathcal{X}$ all belong to the image of~$\phi$.

If $(u_{p}) _{p\in P}$ is any family of elements of $\mathcal{Y}$, then $\langle u_{p}\rangle _{p\in P}$ shall mean the ideal of $\mathcal{Y}$ generated by this family $(u_{p}) _{p\in P}$. Define an ideal $\mathcal{K}$ of~$\mathcal{Y}$ by
\begin{gather}
\mathcal{K}= \langle y_{i,j}+y_{j,i}+\beta \rangle _{(i,j ) \in[n] ^{2}\text{ such that }i\neq j}.\label{pf.prop.symmetry2.defK}
\end{gather}
Clearly, this ideal $\mathcal{K}$ is $S_{n}$-invariant. Hence, the quotient algebra $\mathcal{Y}/\mathcal{K}$ inherits the $S_{n}$-action from~$\mathcal{Y}$.

We are going to show that $\mathcal{X}\cong\mathcal{Y}/\mathcal{K}$.

Let $\pi\colon \mathcal{Y}\rightarrow\mathcal{Y}/\mathcal{K}$ be the canonical projection; this is a surjective $\mathbf{k}$-algebra homomorphism. The $\mathbf{k}$-algebra $\mathcal{Y}$ is generated by the $y_{i,j}$ for all $(i,j) \in[n] ^{2}$ such that $i\neq j$. Hence, the quotient algebra $\mathcal{Y}/\mathcal{K}$ is generated by their projections $\pi(y_{i,j}) $.

It is easy to see that $\phi ( \mathcal{K} ) =0$.\footnote{\textbf{Proof.} It is clearly sufficient to show that $\phi ( y_{i,j}+y_{j,i}+\beta ) =0$ for each $(i,j) \in[n] ^{2}$ such that $i\neq j$ (because of~(\ref{pf.prop.symmetry2.defK})). So let us fix some $(i,j) \in[n] ^{2}$ such that $i\neq j$. We must prove that $\phi ( y_{i,j}+y_{j,i}+\beta ) =0$.

This statement is clearly symmetric in $i$ and $j$; thus, we WLOG assume that $i\leq j$. Hence, $i<j$ (since $i\neq j$). The definition of $\phi$ yields $\phi ( y_{i,j} ) =x_{i,j}$ and $\phi ( y_{j,i} ) =x_{j,i}=-\beta-x_{i,j}$ (by the definition of $x_{j,i}$, since $j>i$). Now, $\phi$ is a $\mathbf{k}$-algebra homomorphism. Thus,
\begin{gather*}
\phi ( y_{i,j}+y_{j,i}+\beta ) =\underbrace{\phi ( y_{i,j} ) }_{=x_{i,j}}+\underbrace{\phi ( y_{j,i} ) }_{=-\beta-x_{i,j}}+\beta=x_{i,j}+ ( -\beta-x_{i,j} ) +\beta=0.
\end{gather*}
This completes our proof.} Hence, the $\mathbf{k}$-algebra homomorphism $\phi$ factors through the projection $\pi\colon \mathcal{Y}\rightarrow \mathcal{Y}/\mathcal{K}$. More precisely: There exists a $\mathbf{k}$-algebra
homomorphism $\phi^{\prime}\colon \mathcal{Y}/\mathcal{K}\rightarrow\mathcal{X}$ satisfying $\phi=\phi^{\prime}\circ\pi$. Consider this $\phi^{\prime}$. Thus, each $(i,j) \in[n] ^{2}$ such that $i\neq j$ satisfies
\begin{gather}
\phi^{\prime} ( \pi(y_{i,j}) ) =\underbrace{(\phi^{\prime}\circ\pi) }_{=\phi}(y_{i,j}) =\phi(y_{i,j}) =x_{i,j} \label{pf.prop.symmetry2.phi'}
\end{gather}
(by the definition of $\phi$).

Define a $\mathbf{k}$-algebra homomorphism $\zeta\colon \mathcal{X}\rightarrow \mathcal{Y}/\mathcal{K}$ by requiring that
\begin{gather}
\zeta(x_{i,j}) =\pi(y_{i,j}) \qquad \text{for each }(i,j) \in[n]^{2}\text{ satisfying }i<j. \label{pf.prop.symmetry2.zeta}
\end{gather}
(This is well-defined by the universal property of $\mathcal{X}$.) Then, it is easy to see that\footnote{\textbf{Proof of (\ref{pf.prop.symmetry2.zeta2}).} Let $(i,j) \in[n] ^{2}$ be such that $i\neq j$. We must prove
that $\zeta(x_{i,j}) =\pi(y_{i,j}) $.

If $i<j$, then this follows immediately from (\ref{pf.prop.symmetry2.zeta}). Thus, we WLOG assume that we don't have $i<j$. Hence, $i\geq j$, so that $i>j$ (since $i\neq j$). In other words, $j<i$. Thus, (\ref{pf.prop.symmetry2.zeta}) (applied to $( j,i) $ instead of $(i,j) $) shows that $\zeta(x_{j,i}) =\pi(y_{j,i}) $.

Notice that $y_{i,j}+y_{j,i}+\beta\in\mathcal{K}$ (by (\ref{pf.prop.symmetry2.defK})), so that $\pi ( y_{i,j}+y_{j,i}+\beta ) =0$ (since $\pi$ is the canonical projection $\mathcal{Y}\rightarrow\mathcal{Y}/\mathcal{K}$).

But the definition of $x_{i,j}$ yields $x_{i,j}=-\beta-x_{j,i}$ (since $i>j$). Applying the map $\zeta$ to both sides of this equality, we obtain
\begin{gather*}
\zeta(x_{i,j}) =\zeta( -\beta-x_{j,i})=-\beta-\underbrace{\zeta(x_{j,i}) }_{=\pi( y_{j,i} ) }\qquad (\text{since }\zeta\text{ is a }\mathbf{k}\text{-algebra homomorphism}) \\
\hphantom{\zeta(x_{i,j})}{} =-\beta-\pi(y_{j,i}) .
\end{gather*}
On the other hand, $\pi$ is a $\mathbf{k}$-algebra homomorphism, so that $\pi( y_{i,j}+y_{j,i}+\beta) =\pi(y_{i,j}) +\pi(y_{j,i}) +\beta$. Thus,
\begin{gather*}
\pi(y_{i,j}) +\pi(y_{j,i}) +\beta=\pi ( y_{i,j}+y_{j,i}+\beta) =0.
\end{gather*}
Hence, $\pi(y_{i,j}) =-\beta-\pi(y_{j,i}) $. Comparing this with $\zeta(x_{i,j}) =-\beta-\pi ( y_{j,i}) $, we obtain $\zeta(x_{i,j}) =\pi(y_{i,j}) $. This completes our proof of~(\ref{pf.prop.symmetry2.zeta2}).}
\begin{gather}
\zeta(x_{i,j}) =\pi(y_{i,j}) \qquad \text{for each }(i,j) \in[n] ^{2}\text{ satisfying }i\neq j. \label{pf.prop.symmetry2.zeta2}
\end{gather}

The equality (\ref{pf.prop.symmetry2.phi'}) shows that the $\mathbf{k}$-algebra homomorphism $\phi^{\prime}\colon \mathcal{Y}/\mathcal{K}\rightarrow \mathcal{X}$ sends the generators $\pi(y_{i,j}) $ of~$\mathcal{Y}/\mathcal{K}$ to the respective generators $x_{i,j}$ of~$\mathcal{X}$. The equality~(\ref{pf.prop.symmetry2.zeta2}) shows that the $\mathbf{k}$-algebra homomorphism $\zeta\colon \mathcal{X}\rightarrow\mathcal{Y}/\mathcal{K}$ sends the generators $x_{i,j}$ of $\mathcal{X}$ back to the respective generators $\pi(y_{i,j}) $ of $\mathcal{Y}/\mathcal{K}$. Hence, these two $\mathbf{k}$-algebra homomorphisms $\phi^{\prime}\colon \mathcal{Y}/\mathcal{K}\rightarrow\mathcal{X}$ and $\zeta\colon \mathcal{X}\rightarrow\mathcal{Y}/\mathcal{K}$ are mutually inverse. Thus, $\phi^{\prime}$ is a $\mathbf{k}$-algebra isomorphism.

Hence, $\mathcal{X}\cong\mathcal{Y}/\mathcal{K}$ as $\mathbf{k} $-algebras. Therefore, the $S_{n}$-action on $\mathcal{Y}/\mathcal{K}$ can be transported to~$\mathcal{X}$. The result is an action of the group~$S_{n}$ on~$\mathcal{X}$ by $\mathbf{k}$-algebra automorphisms satisfying
\begin{gather*}
\sigma\cdot x_{i,j}=x_{\sigma(i) ,\sigma(j)}\qquad \text{for all }\sigma\in S_{n}\ \text{and all pairs }(i,j) \text{ of distinct elements of }[n].
\end{gather*}
Moreover, this is clearly the only such action (because any $\mathbf{k}$-algebra automorphism of $\mathcal{X}$ is determined by its action on the generators~$x_{i,j}$). This proves Proposition~\ref{prop.symmetry2}(a).

(b) If $i$, $j$, $k$ are three distinct elements of $[n] $, then $\sigma\cdot J_{i,j,k}=J_{\sigma(i) ,\sigma(j) ,\sigma ( k ) }$ for each \smash{$\sigma\in S_{n}$} (as follows easily from the definitions of the elements involved). Hence, the action of $S_{n}$ on $\mathcal{X}$ permutes the family $ ( J_{i,j,k} )_{i,j,k\text{ are three distinct elements of }[n] }$. Thus, the ideal generated by this family is $S_{n}$-invariant. Since this ideal is~$\mathcal{J}$ (by Proposition~\ref{prop.symmetry}), we have thus shown that $\mathcal{J}$ is $S_{n}$-invariant. Hence, the quotient $\mathbf{k}$-algebra $\mathcal{X}/\mathcal{J}$ inherits an $S_{n}$-action from $\mathcal{X}$. Proposition~\ref{prop.symmetry2}(b) is thus proven.
\end{proof}

\section[Forkless polynomials and a basis of $\mathcal{X}/\mathcal{J}$]{Forkless polynomials and a basis of $\boldsymbol{\mathcal{X}/\mathcal{J}}$}

\subsection{Statements}

\looseness=-1 We have thus answered one of the major questions about the ideal $\mathcal{J}$; but we have begged perhaps the most obvious one: Can we find a basis of the $\mathbf{k}$-module $\mathcal{X}/\mathcal{J}$? This turns out to be much simpler than the above; the key is to use a different strategy. Instead of reducing polynomials to pathless polynomials, we shall reduce them to \textit{forkless} polynomials, defined as follows:

\begin{Definition} A monomial $\mathfrak{m}\in\mathfrak{M}$ is said to be \textit{forkless} if there exists no triple $(i,j,k) \in[n] ^{3}$ satisfying $i<j<k$ and $x_{i,j}x_{i,k}\,|\,\mathfrak{m}$ (as monomials).

A polynomial $p\in\mathcal{X}$ is said to be \textit{forkless} if it is a $\mathbf{k}$-linear combination of forkless monomials.
\end{Definition}

The following characterization of forkless polynomials is rather obvious:

\begin{Proposition}Let $\mathfrak{m}\in\mathfrak{M}$. Then, the monomial $\mathfrak{m}$ is forkless if and only if there exist a~map $f\colon [n-1] \rightarrow[n] $ and a map $g\colon [n-1] \rightarrow
\mathbb{N}$ such that
\begin{gather*}
 f(i) >i \qquad \text{for each }i\in[n-1] \qquad \text{and}\qquad \mathfrak{m}=\prod _{i\in[n-1] }x_{i,f(i) }^{g(i) }.
\end{gather*}
\end{Proposition}

Now, we claim the following:

\begin{Theorem}\label{thm.forkless.span-uni}Let $p\in\mathcal{X}$. Then, there exists a unique forkless polynomial $q\in\mathcal{X}$ such that $p\equiv q\operatorname{mod}\mathcal{J}$.
\end{Theorem}

\begin{Proposition}\label{prop.forkless.basis}The projections of the forkless monomials $\mathfrak{m}\in\mathfrak{M}$ onto the quotient ring~$\mathcal{X}/\mathcal{J}$ form a~basis of the $\mathbf{k}$-module~$\mathcal{X}/\mathcal{J}$.
\end{Proposition}

\subsection{A reminder on Gr\"{o}bner bases}

Theorem~\ref{thm.forkless.span-uni} and Proposition~\ref{prop.forkless.basis} can be proven using the theory of Gr\"{o}bner bases. See, e.g.,~\cite{BecWei98} for an introduction. Let us outline the argument. We shall use
the following concepts:

\begin{Definition}\label{def.groebner.all}Let $\Xi$ be a set of indeterminates. Let $\mathcal{X}_{\Xi}$ be the polynomial ring $\mathbf{k}\left[ \xi\,|\, \xi\in\Xi\right] $ over $\mathbf{k}$ in these indeterminates. Let $\mathfrak{M}_{\Xi}$ be the set of all monomials in these indeterminates (i.e., the free abelian monoid on the set~$\Xi$).

(For example, if $\Xi=\left\{ x_{i,j}\,|\, (i,j) \in[n]^{2}\text{ satisfying }i<j\right\} $, then $\mathcal{X}_{\Xi}=\mathcal{X}$ and $\mathfrak{M}_{\Xi}=\mathfrak{M}$.)
\begin{enumerate}\itemsep=0pt
\item[(a)] A \textit{term order} on $\mathfrak{M}_{\Xi}$ is a total order on the set $\mathfrak{M}_{\Xi}$ that satisfies the following conditions:
\begin{itemize}\itemsep=0pt
\item Each $\mathfrak{m}\in\mathfrak{M}_{\Xi}$ satisfies $1\leq\mathfrak{m}$ (where $1$ is the trivial monomial in $\mathfrak{M}_{\Xi}$).
\item If $\mathfrak{m}$, $\mathfrak{u}$ and $\mathfrak{v}$ are three elements of $\mathfrak{M}_{\Xi}$ satisfying $\mathfrak{u}\leq\mathfrak{v}$, then $\mathfrak{mu}\leq\mathfrak{mv}$.
\end{itemize}
\item[(b)] If we are given a total order on the set $\Xi$, then we canonically obtain a term order on $\mathfrak{M}_{\Xi}$ defined as follows: For two monomials $\mathfrak{m}=\prod_{\xi\in\Xi}\xi^{m_{\xi}}$ and
$\mathfrak{n}=\prod_{\xi\in\Xi}\xi^{n_{\xi}}$ in $\mathfrak{M}_{\Xi}$, we set $\mathfrak{m}\leq\mathfrak{n}$ if and only if either $\mathfrak{m}=\mathfrak{n}$ or the largest $\xi\in\Xi$ for which $m_{\xi}$ and $n_{\xi}$
differ satisfies $m_{\xi}<n_{\xi}$. This term order is called the \textit{inverse lexicographical order on the set} $\mathfrak{M}_{\Xi}$ \textit{determined by the given total order on }$\Xi$.

\item[(c)] Two monomials $\mathfrak{m}=\prod_{\xi\in\Xi}\xi^{m_{\xi}}$ and $\mathfrak{n}=\prod_{\xi\in\Xi}\xi^{n_{\xi}}$ in $\mathfrak{M}_{\Xi}$ are said to be \textit{non-disjoint} if there exists some $\xi\in\Xi$ satisfying $m_{\xi}>0$ and $n_{\xi}>0$. Otherwise, $\mathfrak{m}$ and $\mathfrak{n}$ are said to be \textit{disjoint}.

From now on, let us assume that some term order on $\mathfrak{M}_{\Xi}$ has been chosen. The next definitions will all rely on this term order.

\item[(d)] If $f\in\mathcal{X}_{\Xi}$ is a nonzero polynomial, then the \textit{head term} of $f$ denotes the largest $\mathfrak{m}\in\mathfrak{M}_{\Xi}$ such that the coefficient of~$\mathfrak{m}$ in $f$ is nonzero. This head term will be denoted by $\operatorname*{HT}(f) $. Furthermore, if $f\in\mathcal{X}_{\Xi}$ is a nonzero polynomial, then the \textit{head coefficient} of $f$ is defined to be the coefficient of~$\operatorname*{HT}(f) $ in $f$; this coefficient will be denoted by $\operatorname*{HC}(f) $.
\item[(e)] A nonzero polynomial $f\in\mathcal{X}_{\Xi}$ is said to be \textit{monic} if its head coefficient~$\operatorname*{HC}(f) $ is~$1$.
\item[(f)] If $\mathfrak{m}=\prod_{\xi\in\Xi}\xi^{m_{\xi}}$ and $\mathfrak{n}=\prod_{\xi\in\Xi}\xi^{n_{\xi}}$ are two monomials in~$\mathfrak{M}_{\Xi}$, then the \textit{lowest common multiple} $\operatorname{lcm}( \mathfrak{m},\mathfrak{n}) $ of~$\mathfrak{m}$ and~$\mathfrak{n}$ is defined to be the monomial $\prod_{\xi \in\Xi}\xi^{\max \{ m_{\xi},n_{\xi} \} }$. (Thus, $\operatorname{lcm} ( \mathfrak{m},\mathfrak{n} ) =\mathfrak{mn}$ if and only if $\mathfrak{m}$ and $\mathfrak{n}$ are disjoint.)
\item[(g)] If $g_{1}$ and $g_{2}$ are two monic polynomials in $\mathcal{X}_{\Xi}$, then the S-\textit{polynomial of}~$g_{1}$ \textit{and}~$g_{2}$ is defined to be the polynomial $\mathfrak{s}_{1}g_{1}-\mathfrak{s}_{2}g_{2}$, where $\mathfrak{s}_{1}$ and $\mathfrak{s}_{2}$ are the unique two monomials satisfying $\mathfrak{s}_{1}\operatorname*{HT} ( g_{1} ) =\mathfrak{s}_{2}\operatorname*{HT} ( g_{2} ) =\operatorname{lcm} ( \operatorname*{HT} ( g_{1} ) ,\operatorname*{HT} (g_{2}) ) $. This S-polynomial is denoted by $\operatorname*{spol} ( g_{1},g_{2} ) $.

From now on, let $G$ be a subset of $\mathcal{X}_{\Xi}$ that consists of monic polynomials.

\item[(h)] We define a binary relation $\underset{G}{\longrightarrow}$ on the set $\mathcal{X}_{\Xi}$ as follows: For two polynomials~$f$ and~$g$ in~$\mathcal{X}_{\Xi}$, we set $f\underset{G}{\longrightarrow}g$ (and say that~$f$ \textit{reduces to }$g$ \textit{modulo}~$G$) if there exists some $p\in G$ and some monomials $\mathfrak{t}\in\mathfrak{M}_{\Xi}$ and $\mathfrak{s} \in\mathfrak{M}_{\Xi}$ with the following properties:
\begin{itemize}\itemsep=0pt
\item The coefficient of $\mathfrak{t}$ in $f$ is $\neq0$.
\item We have $\mathfrak{s}\cdot\operatorname*{HT}(p) =\mathfrak{t}$.
\item If $a$ is the coefficient of $\mathfrak{t}$ in $f$, then $g=f-a\cdot \mathfrak{s}\cdot p$.
\end{itemize}

\item[(i)] We let $\overset{\ast}{\underset{G}{\longrightarrow}}$ denote the reflexive-and-transitive closure of the relation $\underset{G}{\longrightarrow}$.

\item[(j)] We say that a monomial $\mathfrak{m}\in\mathfrak{M}_{\Xi}$ is $G$\textit{-reduced} if it is not divisible by the head term of any element of~$G$.
We say that a polynomial $q\in\mathcal{X}_{\Xi}$ is $G$\textit{-reduced} if $q$ is a $\mathbf{k}$-linear combination of $G$-reduced monomials.

\item[(k)] Let $\mathcal{I}$ be an ideal of $\mathcal{X}_{\Xi}$. The set $G$ is said to be a \textit{Gr\"{o}bner basis} of the ideal $\mathcal{I}$ if and only if the set~$G$ generates~$\mathcal{I}$ and has the following two
equivalent properties:

\begin{itemize}\itemsep=0pt
\item For each $p\in\mathcal{X}_{\Xi}$, there is a unique $G$-reduced $q\in\mathcal{X}_{\Xi}$ such that $p\overset{\ast}{\underset{G}{\longrightarrow}}q$.
\item For each $p\in\mathcal{I}$, we have $p\overset{\ast}{\underset{G}{\longrightarrow}}0$.
\end{itemize}
\end{enumerate}
\end{Definition}

The definition we just gave is modelled after the definitions in \cite[Chapter~5]{BecWei98}; however, there are several minor differences:
\begin{itemize}\itemsep=0pt
\item We use the word ``monomial'' in the same meaning as \cite[Chapter 5]{BecWei98} use the word ``term'' (but not in the same meaning as \cite[Chapter~5]{BecWei98} use the word ``monomial'').

\item We allow $\mathbf{k}$ to be a commutative ring, whereas \cite[Chapter~5]{BecWei98} require $\mathbf{k}$ to be a field. This leads to some complications in the theory of Gr\"{o}bner bases; in particular, not every ideal has a Gr\"{o}bner basis anymore. However, everything \textit{we} are going to use about Gr\"{o}bner bases in this paper is still true in our general setting.

\item We require the elements of the Gr\"{o}bner basis $G$ to be monic, whereas \cite[Chapter~5]{BecWei98} me\-rely assume them to be nonzero polynomials. In this way, we are sacrificing some of the generality of \cite[Chapter~5]{BecWei98} (a sacrifice necessary to ensure that things don't go wrong when~$\mathbf{k}$ is not a field). However, this is not a major loss of generality, since in the situation of \cite[Chapter~5]{BecWei98} the difference between monic polynomials and arbitrary nonzero polynomials is not particularly large (we can scale any nonzero polynomial by a constant scalar to obtain a monic polynomial, and so we can assume the polynomials to be monic in most of the proofs).
\end{itemize}

The following fact is useful even if almost trivial:

\begin{Lemma}\label{lem.groebner.to0}Let $\Xi$, $\mathcal{X}_{\Xi}$ and $\mathfrak{M}_{\Xi}$ be as in Definition~{\rm \ref{def.groebner.all}}. Let $G$ be a subset of~$\mathcal{X}_{\Xi}$ that consists of monic polynomials. Let~$S$ be a finite set. For each $s\in S$, let $g_{s}$ be an element of~$G$, and let $\mathfrak{s}_{s}\in\mathfrak{M}_{\Xi}$ and $a_{s}\in\mathbf{k}$ be arbitrary. Assume that the monomials $\mathfrak{s}_{s}\operatorname*{HT}(g_{s}) $ for all $s\in S$ are distinct. Then, $\sum_{s\in S}a_{s}\mathfrak{s}_{s}g_{s}\overset{\ast}{\underset{G}{\longrightarrow}}0$.
\end{Lemma}

\begin{proof}See \cite{verlong}.
\end{proof}

One of Buchberger's celebrated results is the following proposition, which allows us to verify that a given finite set $G$ is a Gr\"{o}bner basis of an ideal $\mathcal{I}$ using a finite computation:

\begin{Proposition}\label{prop.groebner.buch1}Let $\Xi$, $\mathcal{X}_{\Xi}$ and $\mathfrak{M}_{\Xi}$ be as in Definition~{\rm \ref{def.groebner.all}}. Let $\mathcal{I}$ be an ideal of $\mathcal{X}_{\Xi}$. Let $G$ be a subset of $\mathcal{X}_{\Xi}$ that consists of monic polynomials. Assume that the set $G$ generates $\mathcal{I}$. Then, $G$ is a Gr\"{o}bner basis of $\mathcal{I}$ if and only if it has the following property:
\begin{itemize}\itemsep=0pt
\item If $g_{1}$ and $g_{2}$ are two elements of the set $G$ such that the head terms of $g_{1}$ and $g_{2}$ are non-disjoint, then $\operatorname*{spol} ( g_{1},g_{2}) \overset{\ast}{\underset{G}{\longrightarrow}}0$.
\end{itemize}
\end{Proposition}

Proposition \ref{prop.groebner.buch1} appears (at least in the case when $\mathbf{k}$ is a field) in \cite[Theorem~5.68, (ii)~$\Longleftrightarrow$~(i)]{BecWei98} and \cite[conclusion after the proof of Lemma~1.1.38]{Graaf16}.

We shall also use the following simple fact, known as the ``Macaulay--Buchberger basis theorem'':

\begin{Proposition}\label{prop.groebner.standardbasis}Let $\Xi$, $\mathcal{X}_{\Xi}$ and $\mathfrak{M}_{\Xi}$ be as in Definition \ref{def.groebner.all}. Let $\mathcal{I}$ be an ideal of $\mathcal{X}_{\Xi}$. Let $G$ be a Gr\"{o}bner basis of $\mathcal{I}$. The projections of the $G$-reduced monomials onto the quotient ring $\mathcal{X}_{\Xi}/\mathcal{I}$ form a basis of the $\mathbf{k} $-module $\mathcal{X}_{\Xi}/\mathcal{I}$.
\end{Proposition}

\begin{proof}See \cite[Chapter 5, Section~3, Propositions~1 and~4]{CoLiOs15} or \cite[Th\'{e}or\`{e}me in the section ``Espaces quotients'']{Monass02} or \cite[Theorem~1.2.6]{Sturmf08} or \cite{verlong}.
\end{proof}

\subsection{The proofs}

The main workhorse of the proofs is the following fact:

\begin{Proposition}\label{prop.forkless.groebner}Consider the inverse lexicographical order on the set $\mathfrak{M}$ of monomials determined by
\begin{gather*}
 x_{1,2}>x_{1,3}>\cdots>x_{1,n} >x_{2,3}>x_{2,4}>\cdots>x_{2,n} >\cdots >x_{n-1,n}.
\end{gather*}
Then, the set
\begin{gather}
\big\{ x_{i,k}x_{i,j}-x_{i,j}x_{j,k}+x_{i,k}x_{j,k}+\beta x_{i,k}+\alpha\,|\, (i,j,k) \in[n] ^{3}\text{ satisfying }i<j<k\big\} \label{eq.prop.forkless.groebner.family}
\end{gather}
is a Gr\"{o}bner basis of the ideal $\mathcal{J}$ of $\mathcal{X}$ $($with respect to this order$)$.
\end{Proposition}

\begin{proof}[Proof (sketched)]The elements $x_{i,k}x_{i,j}-x_{i,j}x_{j,k}+x_{i,k}x_{j,k}+\beta x_{i,k}+\alpha$ of the set~(\ref{eq.prop.forkless.groebner.family}) differ from the designated generators
$x_{i,j}x_{j,k}-x_{i,k} ( x_{i,j}+x_{j,k}+\beta ) -\alpha$ of the ideal $\mathcal{J}$ merely by a factor of $-1$ (indeed, $x_{i,k}x_{i,j}-x_{i,j}x_{j,k}+x_{i,k}x_{j,k}+\beta x_{i,k}+\alpha= ( -1 ) ( x_{i,j}x_{j,k}-x_{i,k} ( x_{i,j}+x_{j,k}+\beta ) -\alpha ) $). Thus, they generate the ideal $\mathcal{J}$. Hence, in order to prove that they form a Gr\"{o}bner basis of $\mathcal{J}$, we merely need to show the following claim:

\begin{Claim}\label{Claim1}
Let $g_{1}$ and $g_{2}$ be two elements of the set~\eqref{eq.prop.forkless.groebner.family} such that the head terms of~$g_{1}$ and~$g_{2}$ are non-disjoint. Then, $\operatorname*{spol} ( g_{1},g_{2} ) \overset{\ast}{\underset{G}{\longrightarrow}}0$, where~$G$ is the set~\eqref{eq.prop.forkless.groebner.family}.
\end{Claim}

(Indeed, proving Claim~\ref{Claim1} is sufficient because of Proposition~\ref{prop.groebner.buch1}.)

In order to prove Claim~\ref{Claim1}, we fix two elements $g_{1}$ and $g_{2}$ of the set~(\ref{eq.prop.forkless.groebner.family}) such that the head terms of~$g_{1}$ and~$g_{2}$ are non-disjoint. Thus,
\begin{gather*}
g_{1}=x_{i_{1},k_{1}}x_{i_{1},j_{1}}-x_{i_{1},j_{1}}x_{j_{1},k_{1}} +x_{i_{1},k_{1}}x_{j_{1},k_{1}}+\beta x_{i_{1},k_{1}}+\alpha
\end{gather*}
for some $ ( i_{1},j_{1},k_{1} ) \in[n] ^{3}$ satisfying $i_{1}<j_{1}<k_{1}$, and
\begin{gather*}
g_{2}=x_{i_{2},k_{2}}x_{i_{2},j_{2}}-x_{i_{2},j_{2}}x_{j_{2},k_{2}} +x_{i_{2},k_{2}}x_{j_{2},k_{2}}+\beta x_{i_{2},k_{2}}+\alpha
\end{gather*}
for some $ ( i_{2},j_{2},k_{2} ) \in[n] ^{3}$ satisfying $i_{2}<j_{2}<k_{2}$. Since the head terms $x_{i_{1},k_{1}}x_{i_{1},j_{1}}$ and $x_{i_{2},k_{2}}x_{i_{2},j_{2}}$ of~$g_{1}$ and~$g_{2}$ are non-disjoint, we must have $i_{1}=i_{2}$. Furthermore, one of~$j_{1}$ and~$k_{1}$ must equal one of $j_{2}$ and $k_{2}$ (for the same reason). Thus, there are at most four distinct integers among $i_{1}$, $i_{2}$, $j_{1}$, $j_{2}$, $k_{1}$, $k_{2}$.

We can now finish off Claim~\ref{Claim1} by straightforward computations, after distinguishing several cases based upon which of the numbers~$j_{1}$ and~$k_{1}$ equal which of the numbers $j_{2}$ and $k_{2}$. We WLOG assume that $( i_{1},j_{1},k_{1}) \neq( i_{2},j_{2},k_{2})$
(since otherwise, it is clear that \smash{$\operatorname*{spol}( g_{1},g_{2}) =0\overset{\ast}{\underset{G}{\longrightarrow}}0$}). Thus, there are \textit{exactly} four distinct integers among $i_{1}$, $i_{2}$, $j_{1}$, $j_{2}$, $k_{1}$, $k_{2}$ (since $i_{1}=i_{2}$, since $i_{1}<j_{1}<k_{1}$ and $i_{2}<j_{2}<k_{2}$, and since one of $j_{1}$ and $k_{1}$ equals one of~$j_{2}$ and~$k_{2}$). Let us denote these four integers by $a$, $b$, $c$, $d$ in increasing order (so that $a<b<c<d$). Hence, $i_{1}=a$ (since $i_{1}<j_{1}<k_{1}$ and $i_{2}<j_{2}<k_{2}$), whereas the two pairs $(j_{1},k_{1}) $ and $( j_{2},k_{2}) $ are two of the three pairs $( b,c) $, $( b,d) $ and $( c,d)$ (for the same reason). Hence, $g_{1}$ and $g_{2}$ are two of the three polynomials
\begin{gather*}
 x_{a,c}x_{a,b}-x_{a,b}x_{b,c}+x_{a,c}x_{b,c}+\beta x_{a,c}+\alpha,\\
x_{a,d}x_{a,b}-x_{a,b}x_{b,d}+x_{a,d}x_{b,d}+\beta x_{a,d}+\alpha,\\
x_{a,d}x_{a,c}-x_{a,c}x_{c,d}+x_{a,d}x_{c,d}+\beta x_{a,d}+\alpha.
\end{gather*}
It thus remains to verify that $\operatorname*{spol} ( g_{1},g_{2}) \overset{\ast}{\underset{G}{\longrightarrow}}0$.

Let us do this. Set
\begin{gather*}
u_{1} =x_{a,c}x_{a,b}-x_{a,b}x_{b,c}+x_{a,c}x_{b,c}+\beta x_{a,c}+\alpha,\\
u_{2} =x_{a,d}x_{a,b}-x_{a,b}x_{b,d}+x_{a,d}x_{b,d}+\beta x_{a,d}+\alpha,\\
u_{3} =x_{a,d}x_{a,c}-x_{a,c}x_{c,d}+x_{a,d}x_{c,d}+\beta x_{a,d}+\alpha,\\
u_{4} =x_{b,c}x_{b,d}-x_{b,c}x_{c,d}+x_{b,d}x_{c,d}+\beta x_{b,d}+\alpha.
\end{gather*}
All four polynomials $u_{1}$, $u_{2}$, $u_{3}$, $u_{4}$ belong to $G$. We shall prove that $\operatorname*{spol}( g_{1},g_{2}) \overset{\ast
}{\underset{G}{\longrightarrow}}0$ whenever~$g_{1}$ and~$g_{2}$ are two of the three polynomials $u_{1}$, $u_{2}$, $u_{3}$. In other words, we shall prove that $\operatorname*{spol} ( u_{1},u_{2} ) \overset{\ast }{\underset{G}{\longrightarrow}}0$, $\operatorname*{spol}( u_{1},u_{3}) \overset{\ast}{\underset{G}{\longrightarrow}}0$ and $\operatorname*{spol}( u_{2},u_{3}) \overset{\ast}{\underset{G}{\longrightarrow}}0$.

Start with the neat identity
\begin{gather*}
u_{1} ( x_{a,d}-x_{b,d} ) -u_{2} ( x_{a,c}-x_{b,c} ) -u_{3} ( x_{b,c}-x_{b,d} ) +u_{4} ( x_{a,c}-x_{a,d} )=0.
\end{gather*}
Expanding and bringing $6$ of the $8$ addends on the right hand side, we obtain
\begin{gather*}
x_{a,d}u_{1}-x_{a,c}u_{2}=-x_{b,c}u_{2}-x_{a,c}u_{4}+x_{b,d}u_{1}+x_{b,c}u_{3}+x_{a,d}u_{4}-x_{b,d}u_{3}.
\end{gather*}
Since the monomials
\begin{gather*}
x_{b,c}\operatorname*{HT} ( u_{2} ) ,\quad x_{a,c}\operatorname*{HT} ( u_{4}) ,\quad x_{b,d}\operatorname*{HT}( u_{1}), \quad x_{b,c}\operatorname*{HT}( u_{3}) ,\quad x_{a,d}\operatorname*{HT}( u_{4}) ,\quad x_{b,d}\operatorname*{HT}( u_{3})
\end{gather*}
are distinct, we thus conclude that $x_{a,d}u_{1}-x_{a,c}u_{2}\overset{\ast }{\underset{G}{\longrightarrow}}0$ (by Lemma~\ref{lem.groebner.to0}). In other words, $\operatorname*{spol}( u_{1},u_{2}) \overset{\ast}{\underset{G}{\longrightarrow}}0$ (since $\operatorname*{spol}(
u_{1},u_{2}) =x_{a,d}u_{1}-x_{a,c}u_{2}$).

Next, observe the identity
\begin{gather*}
x_{a,d}u_{1}-x_{a,b}u_{3}=\beta u_{3}-\beta u_{2}-x_{a,b}u_{4}-x_{b,c} u_{2}+x_{b,c}u_{3}+x_{a,d}u_{4}+x_{c,d}u_{1}-x_{c,d}u_{2}.
\end{gather*}
Since the monomials
\begin{gather*}
 \operatorname*{HT}( u_{3}),\quad \operatorname*{HT}(u_{2}), \quad x_{a,b}\operatorname*{HT}( u_{4}), \quad x_{b,c}\operatorname*{HT}( u_{2}), \quad x_{b,c}\operatorname*{HT}(u_{3}) ,\\
 x_{a,d}\operatorname*{HT}( u_{4}) , \quad x_{c,d}\operatorname*{HT}( u_{1}) ,\quad x_{c,d}\operatorname*{HT}( u_{2})
\end{gather*}
are distinct, we can conclude that $x_{a,d}u_{1}-x_{a,b}u_{3}\overset{\ast}{\underset{G}{\longrightarrow}}0$ (by Lemma~\ref{lem.groebner.to0}). In other words, $\operatorname*{spol}( u_{1},u_{3}) \overset{\ast}{\underset{G}{\longrightarrow}}0$.

Finally, the identity we need for $\operatorname*{spol}( u_{2},u_{3}) \overset{\ast}{\underset{G}{\longrightarrow}}0$ is
\begin{gather*}
x_{a,c}u_{2}-x_{a,b}u_{3}=\beta u_{3}-\beta u_{2}-x_{a,b}u_{4}+x_{a,c}u_{4}-x_{b,d}u_{1}+x_{c,d}u_{1}+x_{b,d}u_{3}-x_{c,d}u_{2}.
\end{gather*}
The same distinctness argument works here.

We have thus proven Claim~\ref{Claim1}. Thus, Proposition \ref{prop.forkless.groebner} is proven.
\end{proof}

\begin{Remark}Proposition \ref{prop.forkless.groebner} can be generalized somewhat. Namely, instead of requiring the total order on~$\mathfrak{M}$ to be inverse lexicographic, it suffices to assume that we are given \textit{any} term order on $\mathfrak{M}$ satisfying the following condition: For every $(i,j,k) \in[n] ^{3}$ satisfying $i<j<k$, we have $x_{i,k}>x_{j,k}$ and $x_{i,j}>x_{j,k}$.

In fact, this condition ensures that the head term of the polynomial $x_{i,k}x_{i,j}-x_{i,j}x_{j,k}+x_{i,k}x_{j,k}+\beta x_{i,k}+\alpha$ (for
$(i,j,k) \in[n] ^{3}$ satisfying $i<j<k$) is $x_{i,k}x_{i,j}$; but this is all that was needed from our term order to make the above proof of Proposition~\ref{prop.forkless.groebner} valid.
\end{Remark}

\begin{proof}[Proof of Proposition \ref{prop.forkless.basis} (sketched)]Let $G$ be the set~(\ref{eq.prop.forkless.groebner.family}). Then, Proposition~\ref{prop.forkless.groebner} shows that $G$ is a Gr\"{o}bner basis of the ideal~$\mathcal{J}$ of $\mathcal{X}$ (where $\mathfrak{M}$ is endowed with the term order defined in Proposition~\ref{prop.forkless.groebner}). Hence, Proposition~\ref{prop.groebner.standardbasis} (applied to $\Xi= \{x_{i,j}\,|\, (i,j) \in[n] ^{2}\text{ satisfying }i<j \} $, $\mathcal{X}_{\Xi}=\mathcal{X}$, $\mathfrak{M}_{\Xi}=\mathfrak{M}$ and $\mathcal{I}=\mathcal{J}$) shows that the projections of the $G$-reduced monomials onto the quotient ring $\mathcal{X}/\mathcal{J}$ form a basis of the $\mathbf{k}$-module~$\mathcal{X}/\mathcal{J}$. Since the $G$-reduced monomials are precisely the forkless monomials, this yields
Proposition~\ref{prop.forkless.basis}.
\end{proof}

\begin{proof}[Proof of Theorem \ref{thm.forkless.span-uni}] Theorem~\ref{thm.forkless.span-uni} is merely a restatement of Proposi\-tion~\ref{prop.forkless.basis}.
\end{proof}

\begin{Remark}\label{rmk.forkless.span-uni.existence}Let us notice that the ``existence'' part of Theorem~\ref{thm.forkless.span-uni} can also be proven similarly to how we proved Proposition~\ref{prop.path-red.span}. This time, we need to define a different notion of ``weight'': Instead of defining the weight of a monomial $\mathfrak{m}=\prod_{\substack{(i,j) \in[n] ^{2};\\i<j}}x_{i,j}^{a_{i,j}}$ to be $\operatorname*{weight}\mathfrak{m}=\sum_{\substack{(i,j) \in[n] ^{2};\\i<j}}a_{i,j}(n-j+i) $, we now must define it to be $\operatorname*{weight}\mathfrak{m}=\sum_{\substack{(i,j) \in[n] ^{2};\\i<j}}a_{i,j}( j-i)$.
\end{Remark}

\begin{question}Is there a similarly simple argument for the ``uniqueness'' part?
\end{question}

\subsection{Dimensions}

The $\mathbf{k}$-module $\mathcal{X}/\mathcal{J}$ is free of infinite rank whenever $n\geq2$; indeed, the basis given in Proposition~\ref{prop.forkless.basis} is infinite (for any $k\in\mathbb{N}$, the monomial $x_{1,2}^{k}$ is forkless). However, $\mathcal{X}/\mathcal{J}$ can be equipped with a filtration, whose filtered parts are of finite rank. Namely, recall that the polynomial ring~$\mathcal{X}$ is graded (by total degree) and thus filtered; this filtration is then inherited by its quotient ring $\mathcal{X}/\mathcal{J}$. For each $k\in\mathbb{N}$, we let $(\mathcal{X}/\mathcal{J}) _{\leq k}$ denote the $k$-th part of the filtration on~$\mathcal{X}/\mathcal{J}$ (that is, the projection onto $\mathcal{X}/\mathcal{J}$ of all polynomials $p\in\mathcal{X}$ of total degree $\leq k$). A~moment of thought reveals that the basis of $\mathcal{X}/\mathcal{J}$ given in Proposition~\ref{prop.forkless.basis} is a filtered basis: For each $k\in\mathbb{N}$, the projections of the forkless monomials $\mathfrak{m}\in\mathfrak{M}$ of total degree $\leq k$ onto the quotient ring
$\mathcal{X}/\mathcal{J}$ form a basis of the $\mathbf{k}$-module $( \mathcal{X}/\mathcal{J})_{\leq k}$. This basis is a finite basis, and so its size is a nonnegative integer. What is this integer?

Of course, it suffices to count the forkless monomials $\mathfrak{m}\in\mathfrak{M}$ of total degree $k$ for each $k\in\mathbb{N}$. This is a~relatively easy counting problem using some classical results \cite[Propositions~1.3.7 and~1.3.10]{Stanley-EC1}; the answer is
the following:

\begin{Proposition}\label{prop.forkless.count}For each $k\in\mathbb{N}$, let $f_{n,k}$ be the number of forkless monomials $\mathfrak{m}\in\mathfrak{M}$ of total degree~$k$. Then,
\begin{gather*}
\sum_{k\in\mathbb{N}}f_{n,k}t^{k}=\frac{ ( 1+0t ) (1+1t ) \cdots ( 1+(n-2) t ) }{(1-t) ^{n-1}}
\end{gather*}
$($as formal power series in $\mathbb{Z}[[t]])$.
\end{Proposition}

When $\alpha=0$ and $\beta=0$, the ideal $\mathcal{J}$ of $\mathcal{X}$ is homogeneous. Thus, in this case, the quotient $\mathbf{k}$-algebra $\mathcal{X}/\mathcal{J}$ inherits not only a filtration, but also a grading from $\mathcal{X}$; its Hilbert series (with respect to this grading) is the power series $\sum_{k\in\mathbb{N}}f_{n,k}t^{k}$ of Proposition~\ref{prop.forkless.count}.

\begin{question}The filtration on $\mathcal{X}/\mathcal{J}$ considered above is not the only natural one. Another is the filtration by ``weight'' (as in the proof of Proposition \ref{prop.path-red.span}), or by the alternative notion of ``weight'' mentioned in Remark~\ref{rmk.forkless.span-uni.existence}. What are the analogues of Proposition~\ref{prop.forkless.count} for these filtrations?
\end{question}

\section{Further questions}

Let us finally indicate some further directions of research not mentioned so far.

\subsection[The kernel of $A$]{The kernel of $\boldsymbol{A}$}\label{subsect.Qrat}
\begin{question}\label{quest.Ainj}\quad
\begin{enumerate}\itemsep=0pt
\item[(a)] Is $\mathcal{J}$ the kernel of the map $A\colon \mathcal{X}\rightarrow\mathcal{Q}$ from Definition~\ref{def.A}?

\item[(b)] Consider the polynomial ring $\mathbf{k}[ \widetilde{q}_{1},\widetilde{q}_{2},\ldots,\widetilde{q}_{n}] $ in~$n$ indeterminates $\widetilde{q}_{1},\widetilde{q}_{2},\ldots,\widetilde{q}_{n}$ over $\mathbf{k}$. Let~$\mathcal{Q}_{\operatorname*{rat}}$ denote the localization of this polynomial ring at the multiplicative subset gene\-ra\-ted by all differences of the form $\widetilde{q}_{i}-\widetilde{q}_{j}$ (for $1\leq i<j\leq n$). Then, the morphism $A\colon \mathcal{X}\rightarrow\mathcal{Q}$ factors through a $\mathbf{k}$-algebra homomorphism $\widetilde{A}\colon \mathcal{X} \rightarrow\mathcal{Q}_{\operatorname*{rat}}$ which sends each~$x_{i,j}$ to $-\frac{\widetilde{q}_{i}+\beta+\alpha/\widetilde{q}_{j}}{1-\widetilde{q} _{i}/\widetilde{q}_{j}}=-\frac{\widetilde{q}_{i}\widetilde{q}_{j}+\beta\widetilde{q}_{j}+\alpha}{\widetilde{q}_{j}-\widetilde{q}_{i}} \in\mathcal{Q}_{\operatorname*{rat}}$. Is $\mathcal{J}$ the kernel of this latter homomor\-phism~$\widetilde{A}$?
\end{enumerate}
\end{question}

Parts (a) and (b) of Question \ref{quest.Ainj} are equivalent, since the canonical $\mathbf{k}$-algebra homomorphism $\mathcal{Q}_{\operatorname*{rat}}\rightarrow\mathcal{Q}$ is injective. This question is interesting partly because a positive answer to part~(b) would provide a realization of $\mathcal{X}/\mathcal{J}$ as a subalgebra of a~localized polynomial ring in (only) $n$ indeterminates. This subalgebra would probably not be the whole $\mathcal{Q}_{\operatorname*{rat}}$.

(Perhaps it can be shown -- by some kind of multidimensional residues -- that~$A$ maps the forkless monomials in $\mathcal{X}$ to linearly independent elements of~$\mathcal{Q}$. Such a proof would then immediately yield positive answers to parts~(a) and~(b) of Question~\ref{quest.Ainj} as well as an alternative proof of Theorem~\ref{thm.forkless.span-uni}.)

An approach to Question \ref{quest.Ainj}(b) might begin with finding a basis of the $\mathbf{k}$-module $\mathcal{Q}_{\operatorname*{rat}}$. It
turns out that such a basis is rather easy to construct:

\begin{Proposition}\label{prop.Qrat.basis}In $\mathcal{Q}_{\operatorname*{rat}}$, consider the family of all elements of the form $\prod_{i=1}^{n}g_{i}$, where each $g_{i}$ has either the form $\frac{1}{( \widetilde{q}_{i}-\widetilde{q}_{j}) ^{m}}$ for some $j\in\{ i+1,i+2,\ldots,n\} $ and $m>0$ or the form $\widetilde{q}_{i}^{k}$ for some $k\in\mathbb{N}$. This family is a basis of the $\mathbf{k}$-module $\mathcal{Q}_{\operatorname*{rat}}$.
\end{Proposition}

Notice that this family is similar to the forkless monomials in Proposition~\ref{prop.forkless.basis}, but it is ``larger'' (if we would allow the $g_{i}$ to have the form $\widetilde{q}_{i}^{k}$ only for $k=0$, then we would obtain a~restricted family that would be in an obvious bijection with the forkless monomials).

Proposition \ref{prop.Qrat.basis} is closely related to results by Horiuchi and Terao~\cite{HorTer03,Terao02}; indeed, if~$\mathbf{k}$ is a~field, then $\mathcal{Q}_{\operatorname*{rat}}$ can be regarded as the ring of regular functions on the complement of the braid arrangement in $\mathbf{k}^{n}$, and such functions are what they have studied (although usually not the whole $\mathcal{Q}_{\operatorname*{rat}}$). Notice however that they worked only over fields~$\mathbf{k}$ of characteristic~$0$.

Let us only briefly hint to how Proposition \ref{prop.Qrat.basis} is proven; the details shall be deferred to future work. We can construct $\mathcal{Q}_{\operatorname*{rat}}$ recursively: For any $n>0$, we can first construct the $\mathbf{k}$-algebra~$\mathcal{Q}_{\operatorname*{rat},n-1}$ defined as the localization of the polynomial ring $\mathbf{k}[ \widetilde{q}_{2},\widetilde{q}_{3},\ldots,\widetilde{q}_{n}] $ at the multiplicative subset generated by all differences of the form $\widetilde{q}_{i}-\widetilde{q}_{j}$ (for $2\leq i<j\leq n$); then, $\mathcal{Q}_{\operatorname*{rat}}$ is isomorphic to the localization of the polynomial ring $\mathcal{Q}_{\operatorname*{rat},n-1} [ \widetilde{q}_{1} ] $ at the multiplicative subset generated by all differences of the form $\widetilde{q}_{1}-\widetilde{q_{j}}$ (for $2\leq j\leq n$). Thus, Proposition~\ref{prop.Qrat.basis} can be proven by induction over~$n$, using the following fact:

\begin{Proposition} \label{prop.Qrat.indstep}Let $A$ be a commutative ring. Let $f_{1},f_{2},\ldots,f_{n}$ be $n$ elements of~$A$. Assume that for each $1\leq i<j\leq n$, the element $f_{i}-f_{j}$ of $A$ is invertible. Let~$B$ be the localization of the polynomial ring $A[x] $ at the multiplicative subset generated by all differences of the form $x-f_{j}$ $($for $1\leq j\leq n)$. Then, $B$ is a free $A$-module, with a~basis consisting of the following elements:
\begin{itemize}\itemsep=0pt
\item all elements of the form $\frac{1}{( x-f_{j}) ^{m}}$ for $j\in \{ 1,2,\ldots,n \} $ and $m>0$;
\item all elements of the form $x^{k}$ for $k\in\mathbb{N}$.
\end{itemize}
\end{Proposition}

Proposition \ref{prop.Qrat.indstep} is essentially a form of partial fraction decomposition, saying that any element of $B$ can be uniquely written as an $A$-linear combination of elements of the form $\frac{1}{(x-f_{j}) ^{m}}$ for $j\in \{ 1,2,\ldots,n\} $ and $m>0$, plus a polynomial in $A[x] $. This can be proven by thoroughly analyzing the corresponding proof in the case when $A$ is a field; the
invertibility of the differences $f_{i}-f_{j}$ is actually what is needed here (since it entails that the ideals $ ( x-f_{j} ) A [ x]
$ of $A[x] $ for $j\in \{ 1,2,\ldots,n \} $ are pairwise comaximal).

\subsection[Isomorphisms between $\mathcal{X}/\mathcal{J}$ for different $\alpha$, $\beta$]{Isomorphisms between $\boldsymbol{\mathcal{X}/\mathcal{J}}$ for different $\boldsymbol{\alpha}$, $\boldsymbol{\beta}$}

Let us now rename the ideal $\mathcal{J}$ as $\mathcal{J}_{\beta,\alpha}$, in order to stress its dependence on $\beta$ and $\alpha$.

\begin{question} When are the $\mathbf{k}$-algebras $\mathcal{X}/\mathcal{J}_{\beta,\alpha}$ for different choices of $\alpha$ and $\beta$ isomorphic?
\end{question}

For $n=2$, the answer is clearly ``always'', because $\mathcal{X}/\mathcal{J}_{\beta,\alpha}$ does not depend on $\alpha$ and $\beta$ in this case (in fact, $\mathcal{J}_{\beta,\alpha}=0$ when $n=2$). So the question only becomes interesting for $n\geq3$. The answer may well depend on the base ring $\mathbf{k}$, and it is perhaps reasonable to assume that $\mathbf{k}$ is a field here. It is easy to come up with an example where the $\mathbf{k}$-algebras $\mathcal{X}/\mathcal{J}_{\beta,\alpha}$ for different choices of $\alpha$ and $\beta$ are not isomorphic\footnote{For
example, let $\mathbf{k}$ be a field of characteristic $\neq2$, and let $n=3$. Then, $\mathcal{X}/\mathcal{J}_{\beta,\alpha}$ is the quotient of the polynomial ring $\mathcal{X}=\mathbf{k}[ x_{1,2},x_{1,3},x_{2,3}] $ by the principal ideal generated by $x_{1,2}x_{2,3}-x_{1,3}(
x_{1,2}+x_{2,3}+\beta) -\alpha$.

We claim that any $\mathbf{k}$-algebra $\mathcal{X}/\mathcal{J}_{\beta,\alpha }$ for $4\alpha=\beta^{2}$ is non-isomorphic to any $\mathbf{k}$-algebra $\mathcal{X}/\mathcal{J}_{\beta,\alpha}$ for $4\alpha\neq\beta^{2}$.

To see this, it suffices to show that the $\mathbf{k}$-algebra $\mathcal{X}/\mathcal{J}_{\beta,\alpha}$ has a ``$\mathbf{k}$-valued singular point'' (i.e., a $\mathbf{k}$-algebra homomorphism $\varepsilon\colon \mathcal{X}/\mathcal{J}_{\beta,\alpha}\rightarrow\mathbf{k}$ such that there exist three $\mathbf{k}$-linearly independent $(\varepsilon,\varepsilon) $-derivations $\mathcal{X}/\mathcal{J}_{\beta,\alpha}\rightarrow\mathbf{k}$, where an $( \varepsilon,\varepsilon)$\textit{-derivation} means a~$\mathbf{k}$-linear map $\partial\colon \mathcal{X}/\mathcal{J}_{\beta,\alpha}\rightarrow\mathbf{k}$ satisfying $\partial ( fg ) =\partial(f) \varepsilon(g) +\varepsilon(f) \partial (g) $ for all~$f$,~$g$) if and only if $4\alpha=\beta^{2}$. But this is easily verified.}. The following example should stress that isomorphisms nevertheless can exist:

\begin{Example}\quad
\begin{enumerate}\itemsep=0pt
\item[(a)] Let $\gamma\in\mathbf{k}$. The $\mathbf{k}$-algebra isomorphism $\mathcal{X}\rightarrow\mathcal{X}$, $x_{i,j}\mapsto\gamma-x_{i,j}$ descends to a $\mathbf{k}$-algebra isomorphism $\mathcal{X}/\mathcal{J}_{\beta,\alpha }\rightarrow\mathcal{X}/\mathcal{J}_{\beta+2\gamma,\alpha+\beta\gamma +\gamma^{2}}$.
\item[(b)] Let $\rho\in\mathbf{k}$ be invertible. The $\mathbf{k}$-algebra isomorphism $\mathcal{X}\rightarrow\mathcal{X}$, $x_{i,j}\mapsto\rho x_{i,j}$ descends to a~$\mathbf{k}$-algebra isomorphism $\mathcal{X}/\mathcal{J}_{\rho\beta,\rho^{2}\alpha}\rightarrow\mathcal{X}/\mathcal{J}_{\beta,\alpha}$.
\end{enumerate}
\end{Example}

\subsection{A deformation of the Orlik--Terao algebra?}\label{subsect.arnold}

We have already seen in Section \ref{subsect.Qrat} that the $\mathbf{k}$-algebra $\mathcal{X}/\mathcal{J}$ is closely connected to the localization $\mathcal{Q}_{\operatorname*{rat}}$ from Question~\ref{quest.Ainj}. In the parlance of algebraic geometers, $\mathcal{Q}_{\operatorname*{rat}}$ is the coordinate ring of the complement of the braid arrangement in $\mathbf{k}^{n}$. This complement has been the subject of a classical paper by Arnold~\cite{Arnold71}, which discussed its cohomology ring. Arnold's description of this cohomology ring is remarkably similar to our definition of~$\mathcal{X}/\mathcal{J}$ in the case when $\beta=0$ and $\alpha=0$. Namely, Arnold
considers the exterior (i.e., free anticommutative) algebra~$A(n) $ in~$\binom{n}{2}$ indeterminates $\omega_{i,j}$ for $1\leq i<j\leq n$ up to the relations $\omega_{i,j}\omega_{j,k}+\omega_{j,k}\omega_{k,i}+\omega_{k,i}\omega_{i,j}=0$, where~$\omega_{u,v}$ for $u>v$ is defined to be a synonym for $\omega_{v,u}$. He gives a basis \cite[Corollary~3]{Arnold71} of this $\mathbf{k}$-module $A(n) $, which is almost exactly the same as our basis of forkless monomials for~$\mathcal{X}/\mathcal{J}$ (with the difference, of course, that his monomials are squarefree because they live in an exterior algebra, and that his choice of order is different).

Arnold's algebra $A(n) $ has since been significantly generalized. Namely, to every matroid corresponds an \textit{Orlik--Solomon algebra}~\cite{CorEti01}; this recovers the algebra $A(n) $ when the matroid is the graphical matroid of the complete graph $K_{n}$. Seeing that the subdivision algebra $\mathcal{X}/\mathcal{J}$ can be viewed as a~commutative analogue of $A(n) $, we can thus ask for a similar commutative analogue of an arbitrary Orlik--Solomon algebra.

Such an analogue, too, is known \cite{SchToh15}: it is the \textit{Orlik--Terao algebra} of a finite family of vectors. This generalizes~$\mathcal{X}/\mathcal{J}$ in the case when $\beta=0$ and $\alpha=0$. We may thus regard $\mathcal{X}/\mathcal{J}$ as a deformation of a specific Orlik--Terao algebra, and ask for a generalization:

\begin{question}\quad
\begin{enumerate}\itemsep=0pt
\item[(a)] Can an arbitrary Orlik--Terao algebra be deformed by two parameters~$\beta$ and~$\alpha$, generalizing our $\mathcal{X}/\mathcal{J}$? A deformation by one parameter~$\hbar$ (which we suspect to correspond to our $\mathcal{X}/\mathcal{J}$ for the braid arrangement with $\alpha=0$) has been studied by McBreen and Proudfoot in \cite[Appendix~A.2]{McbPro15} at least in the case of a~unimodular family of vectors.
\item[(b)] Does Theorem \ref{thm.t-red.unique} extend to Orlik--Terao algebras?
\end{enumerate}
\end{question}

\looseness=-1 Note that our basis of forkless monomials for $\mathcal{X}/\mathcal{J}$ can be regarded as an ``\textit{nbc} basis'' in the sense of~\cite{CorEti01} (except that our monomials are not required to be squarefree). Indeed, if we totally order the monomials $x_{i,j}$ in such a~way that $x_{i,j}>x_{u,v}$ whenever $i<u$, then the broken circuits of the graphical matroid of $K_{n}$ are precisely the sets of the form $\{\{i,j\} ,\{i,k\}\} $ for $i<j<k$; but these correspond to the precise monomials $x_{i,j}x_{i,k}$ that a forkless monomial cannot be divisible~by. Proudfoot's and Speyer's \cite[Theorem~4]{ProSpe06} leads to a similar basis for Orlik--Terao algebras of arbitrary hyperplane arrangements, and \cite[Theorem~A.9]{McbPro15} extends this to its one-parameter deformation for unimodular arrangements. We may still ask similar questions about oriented matroids not coming from hyperplane arrangements, and we may also ask for combinatorial proofs. Horiuchi's and Terao's works~\cite{HorTer03} and~\cite{Terao02} seem relevant once again.

We end with an overview of algebras similar to $\mathcal{X}/\mathcal{J}$ that have appeared in the literature, making no claims of completeness. See also the last few paragraphs of the Introduction of~\cite{Kirill16} for a~history of these algebras.
\begin{itemize}\itemsep=0pt
\item As mentioned above, in \cite{Arnold71}, Arnold introduced the noncommutative algebra $A(n) $ with anticommuting generators $\omega_{i,j}$ (for $1\leq i<j\leq n$) and relations\footnote{Anticommutativity of the generators means that $\omega_{i,j}\omega_{u,v}=-\omega_{u,v}\omega_{i,j}$ for all $i<j$ and $u<v$, and that $\omega_{i,j}^{2}=0$ for all $i<j$.}
\begin{gather*}
\omega_{i,j}\omega_{j,k}+\omega_{j,k}\omega_{i,k}+\omega_{i,k}\omega _{i,j}=0\qquad \text{for} \quad 1\leq i<j<k\leq n.
\end{gather*}
This was probably the first algebra of this kind to be defined. Note that the relations can be rewritten in the form%
\begin{gather*}
\omega_{i,j}\omega_{j,k}=-\omega_{i,k} ( \omega_{i,j}+\omega _{j,k} )
\end{gather*}
to reveal the similarity to the generators of $\mathcal{J}$, but Arnold's algebra does not include the two ``deforming'' parameters $\alpha$ and $\beta$ of our $\mathcal{X}/\mathcal{J}$. Arnold showed that $A(n) $ is isomorphic to the (integer) cohomology ring of the space
\begin{gather*}
\big\{ ( z_{1},z_{2},\ldots,z_{n} ) \in\mathbb{C}^{n}\,|\, z_{1},z_{2},\ldots,z_{n}\text{ are distinct}\big\},
\end{gather*}
and found a $\mathbf{k}$-linear basis of $A(n) $. (Note that he has been working with $\mathbf{k}=\mathbb{Z}$, but this clearly yields the
same results for all $\mathbf{k}$.)

This algebra $A(n) $ has later been generalized to the Orlik--Solomon algebra of an arbitrary hyperplane arrangement, and more generally of an arbitrary matroid (see, e.g.,~\cite{Yuzvin01} for an exposition of the arrangement case); the algebra $A(n) $ is recovered by taking the braid arrangement.

\item In \cite{VarGel87}, Gelfand and Varchenko have introduced a commutative counterpart of the Orlik--Solomon algebra of a hyperplane arrangement $S$. This algebra $P$ is generated by the constant function~$1$ and the Heaviside functions of the hyperplanes in the arrangement. If~$S$ is the braid arrangement, then this algebra~$P$ is isomorphic to the algebra~$P(n) $ with generators~$x_{i,j}$ (for $1\leq i<j\leq n$) and relations
\begin{gather*}
x_{i,j}^{2} =x_{i,j} \qquad \text{for} \quad 1\leq i<j\leq n, \qquad \text{and}\\
x_{i,j}x_{j,k} ( x_{i,k}-1 ) - ( x_{i,j}-1 ) (x_{j,k}-1) x_{i,k} =0 \qquad \text{for} \quad 1\leq i<j<k\leq n.
\end{gather*}
(On the nose, they require many more relations, corresponding to all circuits of~$S$; we are using the nontrivial fact that the $3$-circuits suffice.) The latter of these relations rewrites as
\begin{gather*}
x_{i,j}x_{j,k}=x_{i,k} ( x_{i,j}+x_{j,k}-1 ) ,
\end{gather*}
which is exactly one of the generators of $\mathcal{J}$ when $\beta=1$ and $\alpha=0$. However, the additional relations $x_{i,j}^{2}=x_{i,j}$ make their algebra~$P(n) $ finite-dimensional as a $\mathbf{k}$-module (unlike our~$\mathcal{X}/\mathcal{J}$). Gelfand and Varchenko find a basis of~$P$ (at least in the case when $\mathbf{k}=\mathbb{C}$), defined in terms of what they call ``open cycles'' (and is nowadays known as broken circuits). If~$S$ is the braid arrangement, and if an appropriate ordering of the hyperplanes is used, then this basis becomes similar to our basis of forkless monomials (Proposition~\ref{prop.forkless.basis}), except that it only contains the squarefree forkless monomials (as the $x_{i,j}^{2}=x_{i,j}$ relations render all other monomials redundant).

This algebra $P$ can be straightforwardly generalized to arbitrary oriented matroids.

\item The Gelfand--Varchenko algebra $P(n) $ is filtered, and its associated graded algebra is the commutative algebra $Q(n)$ with generators $x_{i,j}$ (for $1\leq i<j\leq n$) and relations
\begin{gather*}
x_{i,j}^{2} =0\qquad \text{for} \quad 1\leq i<j\leq n, \qquad \text{and}\\
x_{i,j}x_{j,k} =x_{i,k} ( x_{i,j}+x_{j,k} ) \qquad \text{for} \quad 1\leq i<j<k\leq n.
\end{gather*}
The latter relations are exactly the generators of $\mathcal{J}$ when $\beta=0$ and $\alpha=0$. This connection is explored, e.g., in Moseley's~\cite{Mosele12} (although he imposes a much larger set of relations); again, this $\mathbf{k}$-algebra is a finite-dimensional $\mathbf{k}$-module with a~``broken circuit'' basis. The algebra~$Q(n) $ also appears (as~$A_{n}$) in Mathieu's \cite[Section~6]{Mathie95}.

Again, this generalizes to an arbitrary hyperplane arrangement, yielding what is called its Artinian Orlik--Terao algebra (the algebra $W(
\mathcal{A}) $ in~\cite{OrlTer94}).

\item In \cite{OrlTer94}, Orlik and Terao assign an algebra $\mathbf{K} \big[\alpha_{\mathcal{A}}^{-1}\big] $ to any hyperplane arrangement
$\mathcal{A}$ in a finite-dimensional $\mathbf{K}$-vector space $V$, where $\mathbf{K}$ is any field. Nowadays known as the (big) Orlik--Terao algebra, it is simply the $\mathbf{K}$-subalgebra of the ring of rational functions on~$V$ generated by the reciprocals of the linear forms whose kernels are the hyperplanes of~$\mathcal{A}$. When~$\mathcal{A}$ is the braid arrangement, this $\mathbf{K}$-algebra is isomorphic to the commutative $\mathbf{K}$-algebra $R(n) $ with generators $x_{i,j}$ (for $1\leq i<j\leq n$) and relations
\begin{gather*}
x_{i,j}x_{j,k}=x_{i,k} ( x_{i,j}+x_{j,k} ) \qquad \text{for} \quad 1\leq i<j<k\leq n.
\end{gather*}
The relations here are exactly the generators of $\mathcal{J}$ when $\beta=0$ and $\alpha=0$. Unlike the previous algebras, this one is no longer finite-dimensional, thus being the closest one so far to $\mathcal{X}/\mathcal{J}$. A basis of the (big) Orlik--Terao algebra has been found by Proudfoot and Speyer in~\cite{ProSpe06}.

\item Kirillov, in \cite[Section~4]{Kirill97}, introduces a noncommutative algebra $\mathcal{G}_{n}$ with generators $[i,j] $ (for $1\leq i<j\leq n$) and relations%
\begin{gather}
[i,j] [j,k] =[j,k] [i,k] +[i,k] [i,j]\qquad \text{for} \quad 1\leq i<j<k\leq n,\label{eq.context.Gn.rel1}\\
[j,k] [i,j] =[i,k] [j,k] +[i,j] [i,k] \qquad \text{for} \quad 1\leq i<j<k\leq n,\label{eq.context.Gn.rel2}\\
[i,j][k,l] = [ k,l] [i,j] \qquad \text{for} \quad i<j \quad \text{and} \quad k<l \quad \text{with} \quad \{i,j\}
\cap \{ k,l \} =\varnothing .\nonumber
\end{gather}
Note that the abelianization of this $\mathcal{G}_{n}$ is $R(n)$. Kirillov states (without proof) a basis of~$\mathcal{G}_{n}$ in
\cite[Theorem~4.3]{Kirill97}, which (under abelianization) transforms into the basis from our Proposition~\ref{prop.forkless.basis}.

\item In \cite[Definition 10.5]{Kirill97}, Kirillov goes on to deform the algebra $\mathcal{G}_{n}$, replacing (\ref{eq.context.Gn.rel1}) and~(\ref{eq.context.Gn.rel2}) by
\begin{gather*}
[i,j] [j,k] =[j,k] [i,k] +[i,k] [i,j] +\beta[i,k] \qquad \text{for} \quad 1\leq i<j<k\leq n,\\
[j,k] [i,j] =[i,k] [j,k] +[i,j] [i,k] +\beta[i,k] \qquad \text{for} \quad 1\leq i<j<k\leq n,
\end{gather*}
where $\beta\in\mathbf{k}$ is fixed. He denotes this algebra by $\mathcal{L}_{n,\beta}$, but leaves its properties to further study.

\item The quasi-classical Yang--Baxter algebra $\mathcal{B}(A_{n}) $ is the noncommutative $\mathbf{k}$-algebra with generators~$x_{i,j}$ (for $1\leq i<j\leq n$) and relations
\begin{gather*}
x_{i,j}x_{j,k} =x_{i,k}x_{i,j}+x_{j,k}x_{i,k}+\beta x_{i,k} \qquad \text{for} \quad 1\leq i<j<k\leq n,\\
x_{i,j}x_{k,l} =x_{k,l}x_{i,j} \qquad \text{for} \quad i<j \quad \text{and} \quad k<l \quad \text{with} \quad \{i,j\}\cap\{ k,l\}=\varnothing.
\end{gather*}
Note that this is not the same as $\mathcal{L}_{n,\beta}$, since the second relation of $\mathcal{L}_{n,\beta}$ is missing here. This algebra
$\mathcal{B}(A_{n}) $ was also introduced by Kirillov (according to~\cite{Meszar09}).

\item M\'{e}sz\'{a}ros, in \cite{Meszar09}, studies the abelianization of $\mathcal{B}(A_{n}) $; this is also the abelianization of~$\mathcal{L}_{n,\beta}$. This is the $\mathbf{k}$-algebra $\mathcal{X}/\mathcal{J}$ for $\alpha=0$.

\item In \cite[Definition 5.1]{Kirill16}, Kirillov starts with two parameters $\alpha,\beta\in\mathbf{k}$ and defines the noncommutative $\mathbf{k}$-algebra $\widehat{\operatorname*{ACYB}}_{n} ( \alpha,\beta ) $, which has generators $x_{i,j}$ (for $1\leq i<j\leq n$) and relations%
\begin{gather*}
x_{i,j}x_{j,k} =x_{i,k}x_{i,j}+x_{j,k}x_{i,k}+\beta x_{i,k}+\alpha \qquad \text{for} \quad 1\leq i<j<k\leq n,\\
x_{i,j}x_{k,l}=x_{k,l}x_{i,j} \qquad \text{for} \quad i<j \quad \text{and} \quad k<l \quad \text{with} \quad \{i,j\}\cap\{k,l\} =\varnothing.
\end{gather*}
This algebra deforms $\mathcal{B}(A_{n}) $; its abelianization is our $\mathcal{X}/\mathcal{J}$.
\end{itemize}

\begin{question}Which of these algebras satisfy an analogue of Theorem~\ref{thm.t-red.unique}?
\end{question}

\subsection{Final questions}

Finally, we pose two lateral but (in our view) equally interesting questions about $\mathcal{X}/\mathcal{J}$.

The first question, suggested by a referee, concerns the geometric background of M\'{e}sz\'{a}ros's work. As mentioned in the Introduction, the algebra $\mathcal{X}/\mathcal{J}$ generalizes M\'{e}sz\'{a}ros's ``subdivision algebra'' $\mathcal{S}(A_{n}) $ from~\cite{Meszar09}. The latter owes its name to a geometric interpretation of the relations $x_{i,j}x_{j,k}=x_{i,k}x_{i,j}+x_{i,k}x_{j,k}-x_{i,k}$ that hold in $\mathcal{X}/\mathcal{J}$ when $\beta=-1$ and $\alpha=0$. For example, if we consider the standard basis $( e_{1},e_{2},\ldots,e_{n})$ of $\mathbb{R}^{n}$, then the cone $ \langle e_{i}-e_{j},e_{j}-e_{k}\rangle _{+}$ (where $\langle u_{1},u_{2},\ldots,u_{p}\rangle _{+}$ means the cone spanned by $p$ vectors $u_{1},u_{2},\ldots,u_{p}$) is the union of the two cones $\langle e_{i}-e_{k},e_{i}-e_{j}\rangle _{+}$ and $\langle e_{i}-e_{k},e_{j}-e_{k}\rangle _{+}$, while the intersection of the latter two cones is $\langle e_{i}-e_{k}\rangle _{+}$. Thus, the indicator
functions of these cones satisfy
\begin{gather*}
\mathbf{1}_{\langle e_{i}-e_{j},e_{j}-e_{k}\rangle _{+}}=\mathbf{1}_{\langle e_{i}-e_{k},e_{i}-e_{j}\rangle _{+}
}+\mathbf{1}_{\langle e_{i}-e_{k},e_{j}-e_{k}\rangle _{+}}-\mathbf{1}_{\langle e_{i}-e_{k}\rangle _{+}}%
\end{gather*}
(with $\mathbf{1}_{P}$ denoting the indicator function of a polyhedron $P$), which is reminiscent of our relation $x_{i,j}x_{j,k} =x_{i,k}x_{i,j}+x_{i,k}x_{j,k}-x_{i,k}$. As M\'{e}sz\'{a}ros showed in~\cite{Meszar09}, this similarity can be used in studying root polytopes.

\begin{question}
Can such a geometric interpretation be given for the relations
\begin{gather*}
x_{i,j} x_{j,k}=x_{i,k} ( x_{i,j}+x_{j,k}+\beta ) +\alpha
\end{gather*}
in $\mathcal{X}/\mathcal{J}$ for general $\alpha$ and $\beta$, or at least for values other than $\beta=-1$ and $\alpha=0$?
\end{question}

The last question, entirely out of left field, asks for a connection to the notion of Rota--Baxter algebras (see, e.g., \cite{Guo09} for a~survey):

\begin{question}Is there anything to the superficial similarity \cite{Grinbe18} of the relation $x_{i,j}x_{j,k}=x_{i,k} ( x_{i,j}+x_{j,k}+\beta ) $ with the axiom of a Rota--Baxter algebra?
\end{question}

\subsection*{Acknowledgments}

The \texttt{SageMath} computer algebra system~\cite{SageMath} was of great service during the development of the results below. Conversations with Nick Early have led me to the ideas in Section~\ref{subsect.arnold}, and Victor Reiner has helped me concretize them. This paper has furthermore profited from enlightening comments by Ricky Liu, Karola M\'{e}sz\'{a}ros, Nicholas Proudfoot, Travis Scrimshaw, Richard Stanley, the anonymous referees and editor.

%\pdfbookmark[1]{References}{ref}
\LastPageEnding

\end{document}